\documentclass{my_m2an}
\usepackage{amsmath,amssymb}
\usepackage[mathcal]{euscript}
\usepackage{tikz}
\usetikzlibrary{patterns}
\newcommand{\bfC}{{\boldsymbol C}}
\newcommand{\bfE}{{\boldsymbol E}}
\newcommand{\bfH}{{\boldsymbol H}}
\newcommand{\bff}{{\boldsymbol f}}
\newcommand{\bfu}{{\boldsymbol u}}
\newcommand{\bfuini}{{\bfu_0}}
\newcommand{\bfv}{{\boldsymbol v}}
\newcommand{\bfw}{{\boldsymbol w}}
\newcommand{\bfe}{{\boldsymbol e}}
\newcommand{\bfn}{\boldsymbol n}
\newcommand{\bfx}{\boldsymbol x}
\newcommand{\bfX}{\boldsymbol X}
\newcommand{\bfvarphi}{{\boldsymbol \varphi}}
\newcommand{\bpsi}{{\boldsymbol \psi}}
\newcommand{\bfnabla}{{\boldsymbol \nabla}}

\newcommand{\dgammax}{\ \mathrm{d}\gamma(\bfx)}
\newcommand{\dx}{\ \mathrm{d}\bfx}
\newcommand{\dt}{\ \mathrm{d} t}

\newcommand{\ui}{u_i} \newcommand{\uj}{u_j}
\newcommand{\vi}{v_i} \newcommand{\vj}{v_j}
\newcommand{\wi}{w_i} 

\newcommand{\mesh}{{\mathcal M}}
\newcommand{\edge}{{\sigma}}
\newcommand{\edgeperp}{{\tau}}
\newcommand{\edges}{{\mathcal E}}
\newcommand{\edgesK}{\edges(K)}

\newcommand{\edgesint}{{\mathcal E}_{\mathrm{int}}}
\newcommand{\edgesext}{{\mathcal E}_{\mathrm{ext}}}
\newcommand{\edgesinti}{{\mathcal E}_{\mathrm{int}}^{(i)}}

\newcommand{\edgesexti}{{\mathcal E}_{\mathrm{ext}}^{(i)}}
\newcommand{\edgesi}{{\edges\ei}}
\newcommand{\edgesj}{{\edges\ej}}
\newcommand{\edged}{\epsilon}
\newcommand{\edgesd}{{\widetilde {\edges}}}
\newcommand{\edgesdi}{{\edgesd^{(i)}}}
\newcommand{\edgesdinti}{{\edgesd^{(i)}_{{\rm int}}}}
\newcommand{\edgesdexti}{{\edgesd^{(i)}_{{\rm ext}}}}
\newcommand{\ei}{^{(i)}}
\newcommand{\ej}{^{(j)}}
\newcommand{\nKedge}{\bfn_{K,\edge}}
\newcommand{\deltat}{\delta t}

\newcommand{\Hmesh}{\bfH_\edges}
\newcommand{\Hmeshzero}{\bfH_{\edges,0}}
\newcommand{\Hmeshmzero}{\bfH_{\edges_m,0}}
\newcommand{\Hmeshnzero}{\bfH_{\edges_n,0}}
\newcommand{\Hmeshi}{H_{\edges^{(i)}}}
\newcommand{\Hmeshizero}{H_{\edges^{(i)},0}}

\newcommand{\dive}{{\mathrm{div}}}
\newcommand{\characteristic}{{1 \! \! 1}}
\newcommand{\diam}{{\mathrm{diam}}}

\newcommand{\llbracket}{\bigl[ \hspace{-0.55ex} |}
\newcommand{\rrbracket}{| \hspace{-0.55ex} \bigr]}
\newcommand{\iinud}{i \in \llbracket 1, d \rrbracket}

\newcommand{\nnn}{{n\in\xN}}
\newcommand{\mnn}{{m\in\xN}}
\newcommand{\nti}{{n\to +\infty}}
\newtheorem{theorem}{Theorem}[section]
\newtheorem{lemma}[theorem]{Lemma}

\theoremstyle{definition}
\newtheorem{definition}[theorem]{Definition}

\theoremstyle{remark}
\newtheorem{remark}[theorem]{Remark}

\begin{document}
\title[Convergence of the MAC scheme for the incompressible Navier-Stokes equations]
{Convergence of the Marker-and-Cell scheme for the incompressible Navier-Stokes equations on non-uniform grids}
\author{T. Gallou\"et}
\address{I2M UMR 7373, Aix-Marseille Universit\'e, CNRS, Ecole Centrale de Marseille. 
39 rue Joliot Curie. 13453 Marseille, France. \\ (raphaele.herbin@univ-amu.fr)}
\author{R. Herbin}
\address{I2M UMR 7373, Aix-Marseille Universit\'e, CNRS, Ecole Centrale de Marseille. 
39 rue Joliot Curie. 13453 Marseille, France. \\ (raphaele.herbin@univ-amu.fr)}
\author{J.-C. Latch\'e}
\address{IRSN, BP 13115, St-Paul-lez-Durance Cedex, France (jean-claude.latche@irsn.fr)}
\author{K. Mallem}
\address{I2M UMR 7373, Aix-Marseille Universit\'e, CNRS, Ecole Centrale de Marseille. 
39 rue Joliot Curie. 13453 Marseille, France.\\ (khadidja.mallem@univ-amu.fr)}
% 65M08: Finite volume methods
% 76N15: Gas dynamics, general
% --
% 65M12: Stability and convergence of numerical methods
% 76N19: Existence, uniqueness and regularity theory
%
%\date{\today}

%
\begin{abstract}
We prove in this paper the convergence of the Marker And Cell (MAC) scheme for the discretization of the steady-state and time-dependent incompressible Navier-Stokes equations in primitive variables, on non-uniform Cartesian grids, without any regularity assumption on the solution. 
{\em A priori} estimates on solutions to the scheme are proven; they yield the existence of discrete solutions and the compactness of sequences of solutions obtained with family of meshes the space step and, for the time-dependent case, the time step of which tend to zero.
We then establish that the limit is a weak solution to the continuous problem. 
\end{abstract}
\maketitle

\hspace{.8cm}{\bf Keywords} {Finite-volume methods, MAC scheme, incompressible Navier-Stokes.}
\\

\begin{center}
\hspace{.8cm}{\it  Communicated by Douglas N. Arnold.}

\end{center}
\tableofcontents
%
%--------------------------------------------------------------------------------------------------------------------------------
%
\section{Introduction}\label{sec1:pbcont}

Let $\Omega$ be an open bounded domain of $\xR^d$ with $d=2$ or $d=3$.
The steady-state incompressible Navier-Stokes equations read:
\begin{subequations} \label{eq:ns}
\begin{align}\label{eq:mass} &
\dive \bar \bfu=0 && \mbox{in } \; \Omega,
\\ \label{eq:qdm} &
-\Delta \bar \bfu+( \bar\bfu\cdot\bfnabla) \bar\bfu+\bfnabla \bar p = \bar \bff && \mbox{in }\ \Omega,
\\ \label{bc} &
\bar \bfu=0 &&\mbox{on } \ \partial\Omega.
\end{align} \end{subequations}
where $\bar\bfu$ stands for the (vector-valued) velocity of the flow, $\bar p$ for the pressure and $\bar \bff$ is a given field of $L^2(\Omega)^d$, and where, for two given vector fields $\bfv= (v_1,\ldots, v_d)$ and $\bfw= (w_1,\ldots, w_d)$, the quantity $(\bfv\cdot\bfnabla) \bfw$ is a vector field whose components are $((\bfv\cdot\bfnabla)\bfw)_i = \sum_{k=1}^d v_k \partial_k \wi$, $\iinud$. 
A weak formulation of Problem \eqref{eq:ns} reads:
\begin{subequations} \label{eq:cont-weak}
\begin{align} \nonumber &
\mbox{Find } (\bar \bfu,\bar p) \in H^1_0(\Omega)^d \times L^2_0(\Omega)
\mbox{ such that, } \forall (\bfv,q) \in H_0^1(\Omega)^d \times L^2_0(\Omega), 
\\ \label{eq:qdm-weak} & \hspace{15ex}
\int_\Omega \bfnabla \bar \bfu: \bfnabla \bfv \dx
+ \int_\Omega ((\bar \bfu\cdot\bfnabla)\bar \bfu)\cdot\bfv \dx -\int_\Omega \bar p\, \dive \bfv \dx
= \int_\Omega \bar \bff \cdot \bfv \dx, 
\\ \label{eq:mass-weak} & \hspace{15ex}
\int_\Omega q \ \dive \bar \bfu \dx =0,
\end{align} \end{subequations}
where $L^2_0(\Omega)$ stands for the subspace of $L^2(\Omega)$ of zero mean-valued functions.

\medskip
The time-dependent Navier-Stokes equations are also considered:
\begin{subequations} \label{eq:ns:ins}
\begin{align} \label{Mass:ins} & 
\dive \bar \bfu =0 && \mbox{ in} \ \Omega \times (0,T),
\\ \label{qdm:ins} &
\partial_t\bar \bfu-\Delta \bar \bfu + (\bar\bfu\cdot\bfnabla)\bar\bfu+\bfnabla \bar p= \bar \bff && \mbox{ in }\ \Omega \times (0,T),
\\ \label{bc:ins} &
\bar \bfu=0 && \mbox{ on }\ \partial\Omega \times (0,T),
\\ \label{initiale} &
\bar \bfu(\bfx,0)=\bfu_0 && \mbox{ in} \ \Omega. 
\end{align} \end{subequations}
This problem is posed for $(\bfx,t)$ in $\Omega \times (0,T)$ where $T\in \xR_{+}^*$; the right-hand side $\bar \bff$ is now a given vector field of $L^2(\Omega \times (0,T))^d$ and the initial datum $\bfu_0$ belongs to the space $\bfE(\Omega)$ of divergence-free functions, defined by:
\[
\bfE(\Omega)=\bigl\{\bfu \in H^1_0(\Omega)^d~;\ \dive \bfu=0 \mbox{ a.e. in } \Omega \bigr\}.
\]
A weak formulation of the transient problem \eqref{eq:ns:ins} reads (see e.g. \cite{boy-06-ele}):
\begin{equation} \label{eq:cont-weak:ins}
\begin{array}{l} \displaystyle
\mbox{Find } \bfu \in L^2(0,T;\bfE(\Omega))\cap L^\infty(0,T;L^2(\Omega)^d)
\mbox{ such that, } \forall \bfv \in L^2(0,T;\bfE(\Omega))\cap C_c^\infty(\Omega \times [0,T)),
\\[1ex] \displaystyle \hspace{7ex}
-\int_0^T \int_\Omega \bar\bfu(\bfx,t)\cdot\partial_t \bfv(\bfx,t) \dx\dt
-\int_\Omega \bfuini (\bfx)\cdot \bfv(\bfx,0) \dx
+\int_0^T \int_\Omega \bfnabla \bar\bfu(\bfx,t): \bfnabla \bfv(\bfx,t) \dx\dt 
\quad \\[1ex] \hfill \displaystyle
+\int_0^T \int_\Omega ((\bar\bfu\cdot\bfnabla)\bar\bfu)(\bfx,t) \cdot \bfv(\bfx,t) \dx\dt 
=\int_0^T \int_\Omega \bar \bff(\bfx,t) \cdot \bfv(\bfx,t) \dx \dt.
\end{array}
\end{equation}

\medskip

The Marker-And-Cell (MAC) scheme, introduced in the middle of the sixties \cite{har-65-num}, is one of the most popular methods \cite{pat-80-num,wes-01-pri} for the approximation of the Navier-Stokes equations in the engineering framework, because of its simplicity, its efficiency and its remarkable mathematical properties.
The aim of this paper is to show, under minimal regularity assumptions on the solution, that sequences of approximate solutions obtained by the discretization of problem \eqref{eq:ns}(resp. \eqref{eq:ns:ins}) by the MAC scheme converge to a solution of \eqref{eq:cont-weak}(resp. \eqref{eq:cont-weak:ins}) as the mesh size (resp. the mesh size and the time step) tends (resp. tend) to 0.

\medskip
For the linear problems, the first error analysis seems to be that of \cite{por-78-err} in the case of the time-dependent Stokes equations on uniform square grids. 
The mathematical analysis of the scheme was performed for the steady-state Stokes equations in \cite{nic-92-ana} for uniform rectangular meshes with $H^2$-regularity assumption on the pressure.
Error estimates for the MAC scheme applied to the Stokes equations have been obtained by viewing the MAC scheme as a mixed finite element method \cite{gir-96-fin,han-98-new} or a divergence conforming DG method \cite{kan-08-div}.
Error estimates for rectangular meshes were also obtained for the related covolume method, see \cite{cho-97-ana} and references therein.
Using the tools that were developed for the finite volume theory \cite{book,eym-07-ana}, an order 1 error estimate for non-uniform meshes was obtained in \cite{bla-99-err}, with order 2 convergence for uniform meshes, under the usual regularity assumptions ($H^2$ for the velocities, $H^1$ for the pressure). 
It was recently shown in \cite{Lij-14-sup} that under higher regularity assumptions ($C^4$ for the velocities and $C^3$ for the pressure) and an additional convergence assumption on the pressure, superconvergence is obtained for non uniform meshes.
Note also that the convergence of the MAC scheme for the Stokes equations with a right-hand side in $H^{-1}(\Omega)$ was proven in \cite{bla-05-con}.

\medskip
Mathematical studies of the MAC scheme for the nonlinear Navier-Stokes equations are scarcer. 
A pioneering work was that of \cite{nic-96-ana} for the steady-state Navier-Stokes equations and for uniform rectangular grids.
More recently, a variant of the MAC scheme was defined on locally refined grids and the convergence proof was performed for both the steady-state and time dependent cases in two or three space dimensions \cite{che-14-ext}.
A MAC-like scheme was also studied for the stationary Stokes and Navier-Stokes equations on two-dimensional Delaunay-Vorono\"{\i} grids \cite{eym-14-tri}.
For the Stokes equations on uniform grids, the scheme given in \cite{che-14-ext} coincides with the usual MAC scheme that is classically used in CFD codes. 
However, for the Navier-Stokes equations, the nonlinear convection term is discretized in \cite{che-14-ext} and \cite{eym-14-tri} in a manner reminiscent of what is sometimes done in the finite element framework (see e.g. \cite{tem-84-nav}), which no longer coincides with the usual MAC scheme, even on uniform rectantular grids; this discretization entails a larger stencil, and numerical experiments \cite{chenier} seem to show that it is not as efficient as the classical MAC scheme.

\medskip
Our purpose here is to analyse the genuine MAC scheme for the steady-state and transient Navier-Stokes equations in primitive variables on a non-uniform rectangular mesh in two or three dimensions, and, as in \cite{che-14-ext}, without any assumption on the data nor on the regularity of the the solutions.
The convergence of a subsequence of approximate solutions to a weak solution of the Navier-Stokes equations is proved for both the steady and unsteady case, which yields as a by product the existence of a weak solution, well known since the work of J. Leray \cite{ler-34-ess}.
In the case where uniqueness of the solution is known, the whole sequence of approximate solutions can be shown to converge, see remarks \ref{rem-uniq-stat} and \ref{rem-uniq-instat}.

\medskip
This paper is organized as follows.
In Section \ref{sec:discop}, the MAC space grid and the discrete operators are introduced.
In particular, the velocity convection operator is approximated so as to be compatible with a discrete continuity equation on the dual cells~; this discretization coincides with the usual discretization on uniform meshes \cite{pat-80-num}, contrary to the scheme of \cite{che-14-ext}.
The MAC scheme for the steady state Navier-Stokes equations and its weak formulation are introduced in Section \ref{sec:steady}.
Velocity and pressure estimates are then obtained, which lead to the compactness of sequences of approximate solutions. 
Any prospective limit is shown to be a weak solution of the continuous problem.
Section \ref{sec:unsteady} is devoted to the unsteady Navier-Stokes equations.
An essential feature of the studied scheme is that the (discrete) kinetic energy remains controlled.
We show the compactness of approximate sequences of solutions thanks to a discrete Aubin-Simon argument, and again conclude that any limit of the approximate velocities is a weak solution of the Navier-Stokes equations, thanks to a passage to the limit in the scheme.
In the case of the unsteady Stokes equations, some additional estimates yield the compactness of sequences of approximate pressures; this entails that the approximate pressure converges to a weak solution of the Stokes equations as the mesh size and time steps tend to 0.
%
%---------------------------------------------------------------------------------------------------------------------------------------------------
%
\section{Space discretization}\label{sec:discop}

Let $\Omega$ be a connected subset of $\xR^d$ consisting in a union of rectangles ($d=2$) or orthogonal parallelepipeds ($d=3$); without loss of generality, the edges (or faces) of these rectangles (or parallelepipeds) are assumed to be orthogonal to the canonical basis vectors, denoted by $(\bfe^{(1)}, \ldots, \bfe^{(d)})$.

\begin{definition}[MAC grid] \label{def:MACgrid}
A discretization of $\Omega$ with a MAC grid, denoted by $\mathcal{D}$, is defined by $\mathcal{D} = (\mesh, \edges)$,
where:
\begin{list}{--}{\itemsep=0.ex \topsep=0.5ex \leftmargin=1.cm \labelwidth=0.7cm \labelsep=0.3cm \itemindent=0.cm}
\item $\mesh$ stands for the primal grid, and consists in a conforming structured partition of $\Omega$ in possibly non uniform rectangles ($d=2$) or rectangular parallelepipeds ($d=3$).
A generic cell of this grid is denoted by $K$, and its mass center by $\bfx_K$.
The pressure is associated to this mesh, and $\mesh$ is also sometimes referred to as "the pressure mesh".

\item The set of all faces of the mesh is denoted by $\edges$; we have $\edges= \edgesint \cup \edgesext$, where $\edgesint$ (resp. $\edgesext$) are the edges of $\edges$ that lie in the interior (resp. on the boundary) of the domain.
The set of faces that are orthogonal to $\bfe\ei$ is denoted by $\edgesi$, for $\iinud$.
We then have $\edgesi= \edgesinti \cup \edgesexti$, where $\edgesinti$ (resp. $\edgesexti$) are the edges of $\edgesi$ that lie in the interior (resp. on the boundary) of the domain.

\medskip
For $\edge\in\edgesint$, we write $\edge = K|L$ if $\edge = \partial K \cap \partial L$.
A dual cell $D_\edge$ associated to a face $\edge \in\edges$ is defined as follows:
\begin{list}{-}{\itemsep=0.ex \topsep=0.ex \leftmargin=1.cm \labelwidth=0.7cm \labelsep=0.1cm \itemindent=0.cm}
\item if $\edge=K|L \in \edgesint$ then $D_\edge = D_{K,\edge}\cup D_{L,\edge}$, where $D_{K,\edge}$ (resp. $D_{L,\edge}$) is the half-part of $K$ (resp. $L$) adjacent to $\edge$ (see Fig. \ref{fig:mesh} for the two-dimensional case); 
\item if $\edge \in \edgesext$ is adjacent to the cell $K$, then $D_\edge=D_{K,\edge}$.
\end{list}
We obtain $d$ partitions of the computational domain $\Omega$ as follows:
\[
\Omega = \cup_{\edge \in \edgesi} D_\edge,\quad \iinud,
\]
and the $i^{th}$ of these partitions is called $i^{th}$ dual mesh, and is associated to the $i^{th}$ velocity component, in a sense which is clarified below.
The set of the faces of the $i^{th}$ dual mesh is denoted by $\edgesdi$ (note that these faces may be orthogonal to any vector of the basis of $\xR^d$ and not only $\bfe\ei$) and is decomposed into the internal and boundary edges: $\edgesdi = \edgesdinti\cup \edgesdexti$.
The dual face separating two duals cells $D_\edge$ and $D_{\edge'}$ is denoted by $\edged=\edge|\edge'$.
\end{list}
\end{definition}

\medskip
To define the scheme, we need some additional notations.
The set of faces of a primal cell $K$ and a dual cell $D_\edge$ are denoted by $\edgesK$ and $\edgesd(D_\edge)$ respectively.
For $\edge \in \edges$, we denote by $\bfx_\edge$ the mass center of $\edge$.
The vector $\bfn_{K,\edge}$ stands for the unit normal vector to $\edge$ outward $K$. 
In some case, we need to specify the orientation of a geometrical quantity with respect to the axis:
\begin{list}{-}{\itemsep=0.ex \topsep=0.5ex \leftmargin=1.cm \labelwidth=0.7cm \labelsep=0.3cm \itemindent=0.cm}
\item a primal cell $K$ will be denoted $K = [\overrightarrow{\edge \edge'}]$ if $\edge, \edge' \in \edgesi \cap \edges(K)$ for some $\iinud$ are such that $(\bfx_{\edge'} - \bfx_\edge) \cdot \bfe\ei >0$;
\item we write $\edge =\overrightarrow{K|L}$ if $\edge \in\edgesi$ and $\overrightarrow{\bfx_K\bfx_L}\cdot \bfe\ei>0$ for some $\iinud$;
\item the dual face $\edged$ separating $D_\edge$ and $D_{\edge'}$ is written $\edged = \overrightarrow{\edge|\edge'}$ if $\overrightarrow{\bfx_\edge \bfx_{\edge'}}\cdot \bfe\ei>0$ for some $\iinud$.
\end{list}
For the definition of the discrete momentum diffusion operator, we associate to any dual face $\edged$ a distance $d_\edged$ as sketched on Figure \ref{fig:mesh}. 
For a dual face $\edged \in \edgesd(D_\edge)$, $\edge \in \edgesi$, $\iinud$, the distance $d_\edged$ is defined by:
\begin{align} \label{distance_duale}
d_\edged=
\begin{cases}
d(\bfx_\edge,\bfx_{\edge'}) & \mbox{if} \ \edged = \edge|\edge' \in \edgesdinti,
\\[1ex]
d(\bfx_\edge,\edged) & \mbox{if} \ \edged \in \edgesdexti,
\end{cases}
\end{align}
where $d(\cdot,\cdot)$ denotes the Euclidean distance in $\xR^d$. 

\begin{figure}[hbt] 
\centering
\begin{tikzpicture}
\fill[green!12!blue!20!white] (0.5,1.5)--(4.5,1.5)--(4.5,4)--(0.5,4)--cycle;
\path node at (0.6,3.5) [anchor= west]{$D_\edge$};

% primal mesh:
\draw[very thin] (0.5,0.5)--(8.,0.5);
\draw[very thin] (0.5,2.5)--(8,2.5);
\draw[very thin] (0.5,5.5)--(8,5.5);
\draw[very thin] (0.5,0.)--(0.5,6.);
\draw[very thin] (4.5,0.)--(4.5,6.);
\draw[very thin] (7.5,0.)--(7.5,6.);
\draw[very thick] (0.5,2.5)--(4.5,2.5);
\draw[very thick] (0.5,5.5)--(4.5,5.5);
\draw[very thick] (4.5,2.5)--(7.5,2.5);

% dual edges:
\draw[very thick, red] (0.5,4)--(4.5,4);
\draw[very thick, green!70!black] (4.5,1.5)--(4.5,4);
\draw[very thick, blue] (0.5,1.5)--(0.5,4);
\draw[very thin] (0.5,1.5)--(4.5,1.5);

\path node at (0.6,5.1) [anchor= west]{$K$};
\path node at (0.6,0.9) [anchor= west]{$L$};
\path node at (3.7,2.7) {$\edge=K|L$};
\path node at (7,2.7) {$\edge''$};
\path node at (2.5,2.5) {$\times$};
\path node at (2.5,5.5) {$\times$};
\path node at (6,2.5) {$\times$};

\path node at (2.5,5.2) {$\bfx_{\edge'}$};
\path node at (2.5,2.2) {$\bfx_\edge$};
\path node at (6.1,2.2) {$\bfx_{\edge''}$};

\path node at (0.5,2.9) [anchor= west]{\textcolor{blue}{$\epsilon_2$}};
\path node at (4.5,2.9) [anchor= west]{\textcolor{green!70!black}{$\epsilon_3$}};
\path node at (3.88,5.3) {$\edge'$};
\path node at (2.5,4.2) {\textcolor{red}{$\epsilon_1=\edge|\edge'$}};
\path node at (0.5,-0.02) [anchor= north]{$\partial\Omega$};
\draw[very thin, green!70!black, <->] (2.5,6.5)--(6,6.5); \path node at (4.25,6.52) [anchor= south]
{\textcolor{green!70!black}{$d_{\epsilon_3}$}};
\draw[very thin, blue, <->] (0.5,6.5)--(2.5,6.5);\path node at (1.5,6.52) [anchor= south]
{\textcolor{blue}{$d_{\epsilon_2}$}};
\draw[very thin, red, <->] (0,2.5)--(0,5.5); \path node at (-0.02,4) [anchor= east]
{\textcolor{red}{$d_{\epsilon_1}$}};
\end{tikzpicture}
\caption{Notations for control volumes and dual cells (in two space dimensions, for the second component of the velocity).}
\label{fig:mesh}
\end{figure}

\medskip
The size $h_\mesh$ and the regularity $\eta_\mesh$ of the mesh are defined by:
\begin{align}\label{sizemesh} &
h_\mesh=\max \bigl\{\diam(K), K\in\mesh \bigr\},
\\ \label{regmesh} & 
\eta_\mesh = \max \Bigl\{ \frac{|\edge|}{|\edge'|},
\ \edge \in \edgesi,\ \edge' \in \edgesj,\ i, j \in \llbracket 1, d\rrbracket,\ i\not= j \Bigr\},
\end{align}
where $|\cdot|$ stands for the $(d-1)$-dimensional measure of a subset of $\xR^{d-1}$ (in the sequel, it is also used to denote the $d$-dimensional measure of a subset of $\xR^d$).

\medskip
The discrete velocity unknowns are associated to the velocity cells and are denoted by $(u_\edge)_{\edge\in\edgesi}$, $\iinud$, while the discrete pressure unknowns are associated to the primal cells and are denoted by $(p_K)_{K\in\mesh}$.
The discrete pressure space $L_\mesh$ is defined as the set of piecewise constant functions over each of the grid cells $K$ of $\mesh$, and the discrete $i^{th}$ velocity space $\Hmeshi$ as the set of piecewise constant functions over each of the grid cells $D_\edge,\ \edge\in\edgesi$.
The set of functions of $L_\mesh$ with zero mean value is denoted by $L_{\mesh,0}$.
As in the continuous case, the Dirichlet boundary conditions are (partly) incorporated into the definition of the velocity spaces, by means of the introduction of the spaces $\Hmeshizero \subset \Hmeshi,\ \iinud$, defined as follows:
\[
\Hmeshizero=\Bigl\{u\in\Hmeshi,\ u(\bfx)=0\ \forall \bfx\in D_\edge,\ \edge \in \edgesexti \Bigr\}.
\]
We then set $\Hmeshzero=\prod_{i=1}^d \Hmeshizero$.
Defining the characteristic function $\characteristic_A$ of any subset $A \subset \Omega$ by $\characteristic_A(\bfx)=1$ if $\bfx \in A$ and $\characteristic_A(\bfx)=0$ otherwise, the $d$ components of a function $\bfu \in \Hmeshzero$ and a function $p \in L_\mesh$ may then be written:
\[
\ui = \sum_{\edge\in \edgesi} u_\edge \characteristic_{D_\edge},\ \iinud \quad \mbox{and} \quad
p = \sum_{K \in \mesh} p_K \characteristic_K.
\]

\medskip
Let us now introduce the discrete operators which are used to write the numerical scheme.

\medskip
{\bf Discrete Laplace operator} --
For $\iinud$, the $i^{th}$ component of the discrete Laplace operator is defined by:
\begin{equation} \label{eq:lap}
\begin{array}{l|l}
-\Delta_\edgesi : \quad
& \quad
\Hmeshizero \longrightarrow \Hmeshizero
\\ & \displaystyle \quad
\ui \longmapsto - \Delta_\edgesi \ui = - \sum_{\edge\in \edgesi} (\Delta u)_\edge \characteristic_{D_\edge},
\mbox{ with } -(\Delta u)_\edge=\frac 1 {|D_\edge|} \sum_{\edged\in\edgesd(D_\edge)} 
\phi_{\edge,\edged}(\ui), \\
& \displaystyle \quad \phi_{\edge,\edged}(\ui)=
\begin{cases} \displaystyle
\frac{|\edged|}{d_\edged} (u_\edge-u_{\edge'}),\text{ if }\ \edged=\edge|\edge' \in \edgesdinti,
\\[2ex] \displaystyle
\frac{|\edged|}{d_\edged} u_\edge,\text{ if } \edged\in\edgesdexti\cap \edgesd(D_\edge),
\end{cases}
\end{array}
\end{equation} 
where $d_\edged$ is defined by \eqref{distance_duale}.
The numerical diffusion flux is conservative:
\begin{equation}
\label{conservdiff}
\phi_{\edge,\edged}(\ui)=-\phi_{\edge',\edged}(\ui),\quad\forall \edged=\edge|\edge'\in\edgesdinti.
\end{equation}
The discrete Laplace operator of the full velocity vector is defined by
\begin{equation}
 \begin{array}{l|l}
-\Delta_\edges:
& 
\Hmeshzero \longrightarrow \Hmeshzero
\\[1ex] &
\bfu \mapsto -\Delta_\edges\bfu = (-\Delta_{\edges^{(1)}} u_1,\ldots, -\Delta_{\edges^{(d)}} u_d).
\end{array}
\end{equation}
Let us now recall the definition of the discrete $H^1_0$-inner product \cite{book}: the $H^1_0$-inner product between $\bfu\in\Hmeshzero$ and $\bfv\in\Hmeshzero$ is obtained by taking, for each dual cell, the inner product of the discrete Laplace operator applied to $\bfu$ by the test function $\bfv$ and integrating over the computational domain.
A simple reordering of the sums (which may be seen as a discrete integration by parts) yields, thanks to the conservativity of the diffusion flux \eqref{conservdiff}:
\begin{multline} \label{ps}
\forall (\bfu, \bfv) \in \Hmeshzero \times \Hmeshzero,\ 
\int_\Omega -\Delta_\edges \bfu \cdot \bfv \dx = [\bfu,\bfv]_{1,\edges,0}=\sum_{i=1}^d [\ui,\vi]_{1,\edgesi,0},
\\
\mbox{with, for }\iinud,\ [\ui,\vi]_{1,\edgesi,0} = 
\sum_{\substack{\edged \in \edgesdinti\\ \edged=\edge|\edge'}}
\frac{|\edged|}{d_\edged}\ (u_\edge-u_{\edge'})\ (v_\edge-v_{\edge'})
+ \sum_{\substack{\edged \in \edgesdexti\\ \edged \in \edgesd(D_\edge)}} \frac{|\edged|}{d_\edged}\ u_\edge\ v_\edge.
\end{multline}
The bilinear forms
\[
\left|\begin{array}{l}
\Hmeshizero \times \Hmeshizero \to \xR
\\[1ex]
(\ui,\vi) \mapsto [\ui,\vi]_{1,\edgesi,0}
\end{array}\right.
\quad \mbox{and} \quad
\left|\begin{array}{l}
\Hmeshzero \times \Hmeshzero \to \xR
\\[1ex]
(\bfu,\bfv) \mapsto [\bfu,\bfv]_{1,\edges,0}
\end{array}\right.
\]
are inner products on $\Hmeshizero$ and $\Hmeshzero$ respectively, which induce the following scalar and vector discrete $H^1_0$ norms:\begin{subequations} \label{norm}
\begin{align} \label{normi} &
\|\ui\|^2_{1,\edgesi,0} = [\ui,\ui]_{1,\edgesi,0}
= \sum_{\substack{\edged \in \edgesdinti \\ \edged=\edge|\edge'}} \frac{|\edged|}{d_\edged}\ (u_\edge-u_{\edge'})^2
+ \sum_{\substack{\edged \in \edgesdexti\\ \edged \in \edgesd(D_\edge)}} \frac{|\edged|}{d_\edged}\ u_\edge^2,
\quad \mbox{for } \iinud,
\\ \label{normfull} & 
\|\bfu\|^2_{1,\edges,0} = [\bfu,\bfu]_{1,\edges,0} = \sum_{i=1}^d \|\ui\|^2_{1,\edgesi,0}.
\end{align}
\end{subequations}
\begin{figure}[tb]
\begin{tikzpicture}[scale=1]
%
% \eth_1 u_1
\draw[-,fill=blue!8] (0.5,6.5)--(3.5,6.5)--(3.5,8.5)--(0.5,8.5); \path(3.,6.7) node[blue!70]{$D_\edged$}; % dualdual cell
\draw[-](0,6.5)--(4,6.5); \draw[-](0,8.5)--(4,8.5); \draw[-](0.5,6)--(0.5,9); \draw[-](3.5,6)--(3.5,9); % primal mesh
\draw[->, very thick](0.1,7.5)--(0.9,7.5); \draw[->, very thick](3.1,7.5)--(3.9,7.5);
\path(0.3,7.7) node{$u_\edge$}; \path(3.3,7.7) node{$u_{\edge'}$};
\draw[-, very thick, blue!50] (2,6.5)--(2,8.5); \path(1.85,7.5) node[blue] {$\edged$}; % dual face
\path(2,5) node[anchor=south]{$(\eth_1 u_1)_{D_\edged} = \dfrac{u_{\edge'} - u_\edge}{d_\edged}$};
%
% no external face
\path(10,8.5) node[anchor=north]{\begin{minipage}{0.5\textwidth}
Note that this definition is still valid if $\edge$ or $\edge'$ are external faces (in which case the corresponding velocity is equal to zero).
The volumes $(D_\edged, \mbox{ for } \edged \in \edgesd^{(1)} \mbox{ and } \edged \mbox{ orthogonal to } \bfe^{(1)})$ thus form a partion of $\Omega$ and the definition is complete.\\
Note also that, in the present case, $D_\edged$ is also a primal cell.
\end{minipage}};
%
% \eth_2 u_1, internal face
\draw[-, blue!8, fill=blue!8] (1.5,1)--(3.25,1)--(3.25,3)--(1.5,3); \path(1.8,2.7) node[blue!70]{$D_\edged$}; % dualdual cell
\draw[-](0,0)--(4.5,0); \draw[-](0,2)--(4.5,2); \draw[-](0,4)--(4.5,4);
\draw[-](0.5,-0.5)--(0.5,4.5); \draw[-](2.5,-0.5)--(2.5,4.5); \draw[-](4,-0.5)--(4,4.5); % primal mesh
\draw[->, very thick](2.1,1)--(2.9,1); \draw[->, very thick](2.1,3)--(2.9,3);
\path(2.3,1.2) node{$u_\edge$}; \path(2.3,3.2) node{$u_{\edge'}$};
\draw[-, very thick, blue!50] (1.5,2)--(3.25,2); \path(2.2,2.15) node[blue] {$\edged$}; % dual face
\path(2.25,-1.2) node{$(\eth_2 u_1)_{D_\edged} = \dfrac{u_{\edge'} - u_\edge}{d_\edged}$};
%
% \eth_2 u_1, external face (up)
\draw[-, blue!8, fill=blue!8] (6.5,1)--(8.25,1)--(8.25,2)--(6.5,2); \path(6.8,1.3) node[blue!70]{$D_\edged$}; % dualdual cell
\draw[-](5,0)--(9.5,0); \draw[-, very thick](5,2)--(9.5,2);
\draw[-](5.5,-0.5)--(5.5,2); \draw[-](7.5,-0.5)--(7.5,2); \draw[-](9,-0.5)--(9,2); % primal mesh
\draw[->, very thick](7.1,1)--(7.9,1); \path(7.3,1.2) node{$u_\edge$};
\draw[-, very thick, blue!50] (6.5,2)--(8.25,2); \path(7.5,2.15) node[blue] {$\edged$}; % dual face
\path(7.25,-1.2) node{$(\eth_2 u_1)_{D_\edged} = \dfrac{- u_\edge}{d_\edged}$};
%
% \eth_2 u_1, external face (down)
\draw[-, blue!8, fill=blue!8] (11.5,0)--(13.25,0)--(13.25,1)--(11.5,1); \path(11.8,0.7) node[blue!70]{$D_\edged$}; % dualdual cell
\draw[-, very thick](10,0)--(14.5,0); \draw[-](10,2)--(14.5,2);
\draw[-](10.5,0)--(10.5,2.5); \draw[-](12.5,0)--(12.5,2.5); \draw[-](14,0)--(14,2.5); % primal mesh
\draw[->, very thick](12.1,1)--(12.9,1); \path(12.3,1.2) node{$u_\edge$};
\draw[-, very thick, blue!50] (11.5,0)--(13.25,0); \path(12.5,-0.15) node[blue] {$\edged$}; % dual face
\path(12.25,-1.2) node{$(\eth_2 u_1)_{D_\edged} = \dfrac{u_\edge}{d_\edged}$};
\end{tikzpicture}
\caption{Notations for the definition of the partial space derivatives of the first component of the velocity, in two space dimensions.}
\label{fig:gradient} \end{figure}
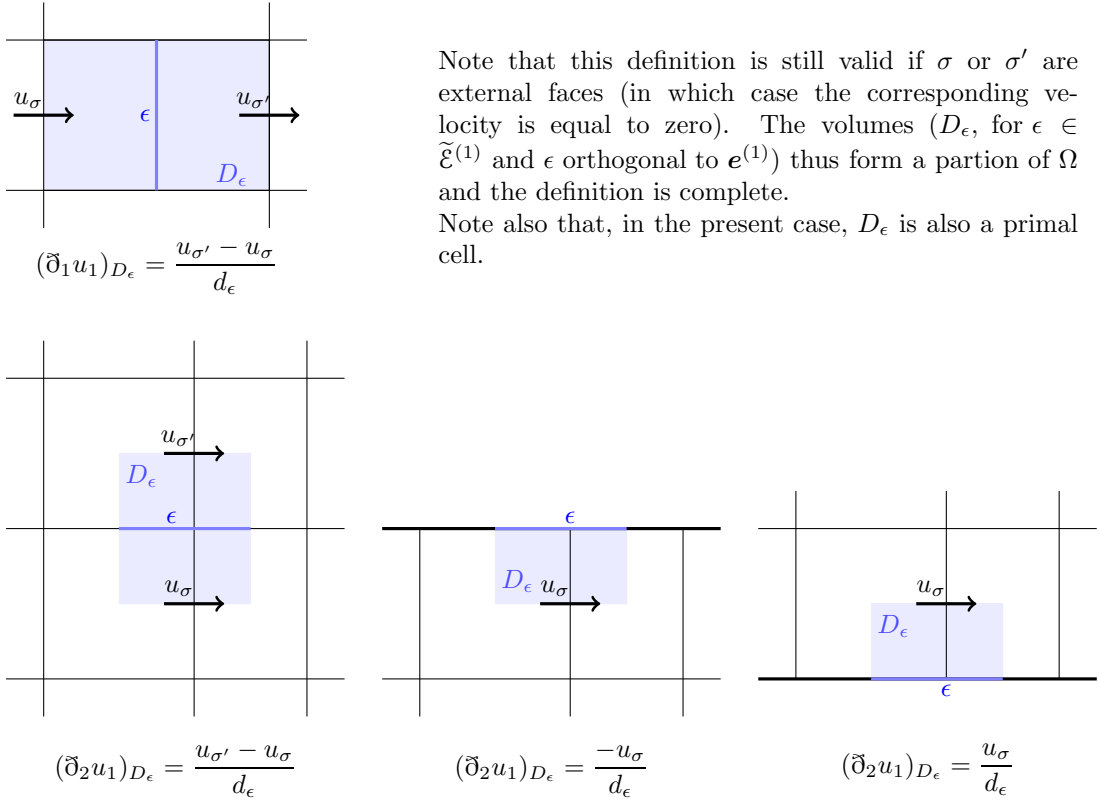
This inner product may also be formulated as the $L^2$-inner product of discrete gradients.
To this purpose, we introduce $d \times d$ new partitions of the domain $\Omega$, where the $(i,j)^{th}$ partition consists in an union of rectangles ($d=2$) or orthogonal parallelepipeds ($d=3$) associated to the dual faces orthogonal to $\bfe\ej$ of the dual mesh $\edgesdi$ for the $i^{th}$ component of the velocity.
This $(i,j)^{th}$ partition reads:
\[
(D_\edged)_{\edged \in \edgesdi, \edged \perp \bfe\ej},\mbox{ with }
\left|\begin{array}{ll}
D_\edged = \edged \times [\bfx_\edge\, \bfx_{\edge'}]
&
\mbox{ if } \edged \mbox{ lies inside } \Omega,\ \edged = \edge|\edge',
\\[1ex]
D_\edged = \edged \times [\bfx_\edge\, \bfx_{\edge,\edged}]
&
\mbox{ if } \edged \mbox{ lies on } \partial \Omega,\ \edged \in \edgesd(D_\edge),
\end{array}\right.
\]
where $\bfx_{\edge,\edged}$ is defined as the orthogonal projection of $\bfx_\edge$ on $\edged$ (which is also, in two space dimensions, the vertex of $\edge$ lying on $\edged$).
The discrete derivative $\eth_j \ui$ is defined on the $(i,j)^{th}$ partition and reads:
\begin{equation} \label{eq:partial-v}
\begin{array}{lll}
\mbox{ if } \edged \mbox{ lies inside } \Omega,\ \edged = \overrightarrow{\edge|\edge'},
\quad &
(\eth_j \ui)_{D_\edged} = \dfrac{u_{\edge'} - u_\edge}{d_\edged},
\\
\mbox{ if } \edged \mbox{ lies on } \partial \Omega,\ \edged \in \edgesd(D_\edge),
\quad &
(\eth_j \ui)_{D_\edged} = \dfrac{- u_\edge}{d_\edged}\ \overrightarrow{\bfx_\edge\bfx_{\edge,\edged}} \cdot \bfe\ej,
\end{array}
\end{equation}
with $d_\edged$ defined by \eqref{distance_duale}.
These definitions are illustrated on Figure \ref{fig:gradient}.
Note that some of these partitions are the same: the $(i,j)^{th}$ partition coincide with the $(j,i)^{th}$ and $(i,i)^{th}$ partitions are the same for $\iinud$.
In addition, these latters also coincide with the primal mesh: for any sub-volume $D_\edged$ of such a partition, there is $K\in\mesh$ such that $D_\edged=K$, and we may thus write equivalently $(\eth_i \ui)_{D_\edged}$ or $(\eth_i \ui)_K$.
We choose this latter notation in the definition of the discrete divergence below for the sake of consistency, since, if we adopt a variational point of view for the description of the scheme, the discrete velocity divergence has to belong (and indeed does belong) to the space of discrete pressures (see Sections \ref{sec:steady} and \ref{sec:unsteady} below for a varitional form of the scheme, in the steady and time-dependent case, respectively).
The discrete discrete gradient of each velocity component $\ui$ may now be defined as:
\begin{equation}
\bfnabla_{\edgesdi} \ui = (\eth_1 \ui, \ldots, \eth_d \ui) \mbox{ with }
\eth_j \ui = \sum_{\substack{\edged \in \edgesdi \\ \edged \perp \bfe\ej}} (\eth_j \ui)_{D_\edged}\ \characteristic_{D_\edged}.
\end{equation}
With this definition, it is easily seen that 
\begin{equation}\label{gradient-and-innerproduct}
\int_\Omega \bfnabla_\edgesdi \psi \cdot \bfnabla_\edgesdi \chi \dx = [\psi,\chi]_{1,\edgesi,0}, \mbox{ for } \psi, \chi \in \Hmeshizero,
\mbox{ and }\iinud.
\end{equation}
If we extend this definition to the velocity vector by 
\[
\bfnabla_\edgesd \bfu= (\bfnabla_{\edgesd^{(1)}} u_1, \ldots, \bfnabla_{\edgesd^{(d)}} u_d),
\]
we get
\[
\int_\Omega \bfnabla_\edgesd \bfu : \bfnabla_\edgesd \bfv \dx = [\bfu,\bfv]_{1,\edges,0}.
\]
This operator satisfies the following consistency result.

\begin{lemma}[Consistency of the discrete partial derivatives of the velocity] \label{lem:grad_v_cons}
Let $\Pi_\edges$ be an interpolation operator from $C_c^\infty(\Omega)^d$ to $\Hmeshzero$ such that, for any $\bfvarphi=(\varphi_1,\cdots,\varphi_d) \in C_c^\infty(\Omega)^d$, there exists $C_\bfvarphi \ge 0$ depending only on $\bfvarphi$ such that
\begin{multline}\label{piedges}
\Pi_\edges \bfvarphi=(\Pi_{\edges^{(1)}} \varphi_1, \cdots, \Pi_{\edges^{(d)}} \varphi_d)
\in H_{\edges^{(1)},0} \times \cdots \times H_{\edges^{(d)},0},\mbox{ where }
\\
\bigl|(\Pi_\edgesi \varphi_i)_\edge - \varphi_i(\bfx_\edge)\bigr| \le C_{\bfvarphi}\ h_\mesh^2,
\mbox{ for } \edge\in\edgesi,\ \iinud.
\end{multline}
Let $\eta_\mesh$ be the parameter measuring the regularity of the mesh defined by \eqref{sizemesh}.
Then there exists $C_{\bfvarphi,\eta_\mesh} \ge 0$, only depending in a non-decreasing way on $\eta_\mesh$, such that 
\[
|\eth_j \Pi_\edgesi \varphi_i (\bfx) - \partial_j \varphi_i (\bfx)| \le C_{\bfvarphi,\eta}\ h_\mesh \mbox{ for a.e. }\bfx \in \Omega
\mbox{ and for } i,j \in \llbracket 1, d \rrbracket.
\]
As a consequence, if $ (\mesh_m,\edges_m)_{m\in\xN}$ is a sequence of MAC grids whose regularity is bounded and whose size tends to 0 as $m$ tends to $+\infty$, then $\bfnabla_{\edgesd_m}\,(\Pi_{\edges_m} \bfvarphi) \to \bfnabla \bfvarphi$ uniformly as $m \to +\infty$.
\end{lemma}
%
% -------------------------------------------------------------------
%
\bigskip 
{\bf Discrete divergence and gradient operators} --
The discrete divergence operator $\dive_\mesh$ is defined by:
\begin{align} \label{eq:div} &
\begin{array}{l| l} \displaystyle
\dive_\mesh:
\quad & \quad
\Hmeshzero \longrightarrow L_{\mesh,0}
\\[1ex] & \displaystyle \quad
\bfu \longmapsto \dive_\mesh\, \bfu = \sum_{K\in\mesh} \frac 1 {|K|} \sum_{\edge\in\edges(K)} |\edge|\ u_{K,\edge} \ \characteristic_K,
\end{array}
\\ & \mbox{with } u_{K,\edge} = u_\edge \nKedge \cdot \bfe\ei \mbox{ for } \edge \in \edgesi \cap \edgesK,\ \iinud.
\end{align}
Note that the numerical flux is conservative, \ie
\begin{equation} \label{conservativite}
u_{K,\edge} =-u_{L,\edge} ,\quad\forall \edge=K|L\in\edgesint .
\end{equation}
We can now define the discrete divergence-free velocity space:
\[
\bfE_\edges(\Omega)= \bigl\{\bfu\in\Hmeshzero~;~\dive_\mesh\, \bfu=0 \bigr\}.
\] 
The discrete divergence of $\bfu = (u_1, \ldots, u_d) \in \Hmeshzero$ may also be written as
\[
\dive_\mesh\, \bfu = \sum_{i=1}^d (\eth_i \ui)_K \characteristic_K,
\]
where the discrete derivative $(\eth_i \ui)_K$ is defined by Relation \eqref{eq:partial-v}.

\medskip
The gradient (which applies to the pressure) in the discrete momentum balance equation is built as the dual operator of the discrete divergence, and reads:
\begin{equation}\label{eq:grad}
\begin{array}{l|l}
\bfnabla_\edges:\quad
& \quad
L_\mesh \longrightarrow \Hmeshzero 
\\[1ex] & \displaystyle \quad
p \longmapsto \bfnabla_\edges p = (\eth_{1} p, \ldots, \eth_d p), 
\end{array}
\end{equation}
where $\eth_i p \in \Hmeshizero$ is the discrete derivative of $p$ in the $i$-th direction, defined by: 
\begin{equation} \label{discderive}
\eth_i p(\bfx) = \frac{|\edge|}{|D_\edge|}\ (p_L - p_K)\, \quad \forall \bfx\in D_\edge,
\mbox{ for } \edge=\overrightarrow{K|L} \in \edgesinti, \ \iinud.
\end{equation}
Note that, in fact, the discrete gradient of a function of $L_\mesh$ should only be defined on the internal faces, and does not need to be defined on the external faces; it is chosen to be in $\Hmeshzero$ (that is zero on the external faces) for the sake of simplicity. 
Again, the definition of the discrete derivatives of the pressure on the MAC grid is consistent in the sense made precise in the following lemma.

\begin{lemma}[Discrete gradient consistency]\label{lem:consgrad}
Let $\Pi_\mesh $ be an interpolation operator from $C_c^\infty(\Omega)$ to $L_\mesh$ such that, for any $\psi \in C_c^\infty(\Omega)$, there exists $C_\psi\ge 0$ depending only on $\psi$ such that
\begin{equation} \label{pimesh}
|(\Pi_\mesh \psi)_K - \psi(\bfx_K)| \le C_\psi\ h_\mesh^2, \mbox{ for } K \in \mesh.
\end{equation}
then there exists $C_{\psi,\eta_\mesh}\ge 0$ depending only on $\psi$ and, in a non-decreasing way, on $\eta_\mesh$, such that
\[
|\eth_i \Pi_\mesh \psi (\bfx) - \partial_i \psi(\bfx)| \le C_{\psi,\eta}\ h_\mesh, \mbox{ for a.e. }\bfx \in \Omega
\mbox{ and for } \iinud.
\]
\end{lemma}

\begin{lemma}[Discrete $\dive-\bfnabla$ duality]
\label{lem:duality}
Let $\ q\in L_\mesh$ and $\bfv\in\Hmeshzero$ then:
\begin{equation}
\int_\Omega q \ \dive_\mesh\,\bfv \dx +\int_\Omega \bfnabla_\edges q\cdot \bfv \dx =0 \label{Ndiscret}.
\end{equation}
\end{lemma}
\begin{proof}
Let $\ q\in L_\mesh$ and $\bfv\in\Hmeshzero$. 
By the definition \eqref{eq:div} of the discrete divergence operator and thanks to the conservativity \eqref{conservativite} of the flux:
\[
\int_\Omega q \ \dive_\mesh\,\bfv \dx
=\sum_{K\in\mesh} q_K \sum_{\edge \in \edges(K)} |\edge|\ v_{K,\edge}
=\sum_{\edge\in\edgesint, \edge=K|L} |\edge|\ (q_K - q_L)\, v_{K,\edge}.
\]
Therefore, by the definition \eqref{discderive} of the discrete derivative of $q$, 
\[
 \int_\Omega q \ \dive_\mesh\,\bfv \dx 
 =-\sum_{i=1}^d\ \sum_{\edge\in\edgesi} |D_\edge|\ v_\edge\, \eth_i q 
 =-\int_\Omega \bfnabla_\edges q \cdot \bfv \dx,
\]
which concludes the proof.
\end{proof}
%
%
% -------------------------------------------------------------------
%
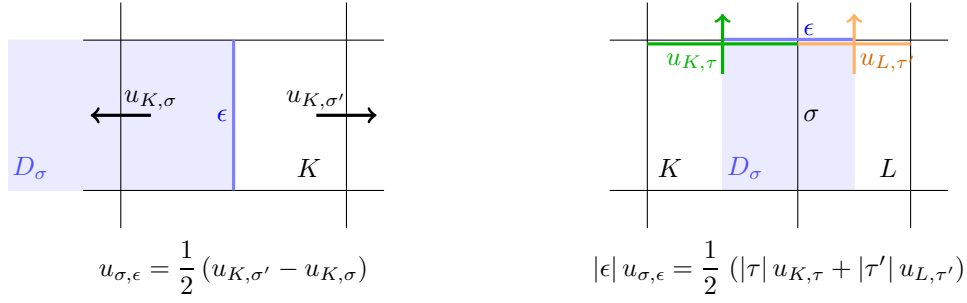
\begin{figure}[tb]
\begin{tikzpicture}[scale=1]
%
% u_1, dual faces orthogonal to e^(1)
\draw[-, blue!8, fill=blue!8] (-1,0.5)--(2,0.5)--(2,2.5)--(-1,2.5); \path(-0.7,0.8) node[blue!70]{$D_\edge$}; % dual cell
\path(3.,0.8) node{$K$};
\draw[-](0,0.5)--(4,0.5); \draw[-](0,2.5)--(4,2.5); \draw[-](0.5,0)--(0.5,3); \draw[-](3.5,0)--(3.5,3); % primal mesh
\draw[<-, very thick](0.1,1.5)--(0.9,1.5); \draw[->, very thick](3.1,1.5)--(3.9,1.5);
\path(0.9,1.7) node{$u_{K,\edge}$}; \path(3.1,1.7) node{$u_{K,\edge'}$};
\draw[-, very thick, blue!50] (2,0.5)--(2,2.5); \path(1.85,1.5) node[blue] {$\edged$}; % dual face
\path(2,-1) node[anchor=south]{$u_{\edge,\edged} = \dfrac 1 2 \,(u_{K,\edge'} - u_{K,\edge})$};

%
% u_1, dual faces orthogonal to e^(2)
\draw[-, blue!8, fill=blue!8] (8.5,0.5)--(10.25,0.5)--(10.25,2.5)--(8.5,2.5); \path(8.8,0.8) node[blue!70]{$D_\edge$};
\path(7.8,0.8) node{$K$}; \path(10.7,0.8) node{$L$};
\draw[-](7,0.5)--(11.5,0.5); \draw[-](7,2.5)--(11.5,2.5);
\draw[-](7.5,0)--(7.5,3); \draw[-](9.5,0)--(9.5,3); \draw[-](11,0)--(11,3); % primal mesh
\path(9.45,1.5) node[anchor=west]{$\edge$};
\draw[-, very thick, blue!50] (8.5,2.51)--(10.25,2.51); \path(9.45,2.65) node[blue, anchor=west] {$\edged$}; % dual face
\draw[-, very thick, green!70!black] (7.5,2.45)--(9.5,2.45); % primal face of K
\draw[-, very thick, orange!60] (9.5,2.45)--(11,2.45); % primal face of L
\draw[->, very thick, green!70!black](8.5,2.05)--(8.5,2.85); \draw[->, very thick,orange!60](10.25,2.05)--(10.25,2.85);
\path(8.1,2.2) node[green!70!black]{$u_{K,\edgeperp}$}; \path(10.7,2.2) node[orange!80!black]{$u_{L,\edgeperp'}$};
\path(9.25,-1) node[anchor=south]{$|\edged|\, u_{\edge,\edged} = \dfrac 1 2\ (|\edgeperp|\, u_{K,\edgeperp} + |\edgeperp'|\, u_{L,\edgeperp'})$};
\end{tikzpicture}
\caption{Mass fluxes in the definition of the convection operator for the primal component of the velocity, in two space dimensions.}
\label{fig:convection} \end{figure}

\bigskip 
{\bf Discrete convection operator} --
Let us consider the momentum equation \eqref{eq:qdm} for the $i^{th}$ component of the velocity, and integrate it on a dual cell $D_\edge$, $\edge \in \edgesi$. 
By the Stokes formula, we then need to discretize $\sum_{\edged \in \edgesd(D_\edge)} \int_\edged \ui\, \bfu \cdot \bfn_{\edge,\edged}\dgammax,$ where $\bfn_{\edge,\edged}$ denotes the unit normal vector to $\edged$ outward $D_\edge$ and $\dgammax$ denotes the $d-1$-dimensional Lebesgue measure.
For $\edged = \edge | \edge'$, the convection flux $\int_\edged \ui \bfu \cdot \bfn_{\edge,\edged}\dgammax$ is approximated by $|\edged|\, u_{\edge,\edged}\, u_\edged^\ast$; usually, $u_\edged^\ast$ is chosen as the mean value of the two unknowns $u_\edge$ and $u_{\edge'}$. 
In some situations (high Reynolds number for instance), an upwind choice may be preferred. 
The two possible choices that will be considered for $u_\edged$ are thus:
\begin{equation} \label{dual:unknown}
u_\edged^\ast= u_\edged^{\mathrm{c}}=\frac{u_\edge+u_{\edge'}} 2 \mbox{ (centred choice, $\ast$=c) or }
u_\edged^\ast = u_\edged^{\mathrm{up}} =
\begin{cases} u_\edge \mbox{ if } u_{\edge, \edged} \ge 0, \\ u_{\edge'} \mbox{ otherwise,} \end{cases}
\mbox{ (upwind choice, $\ast$=up).}
\end{equation}
The quantity $|\edged|\, u_{\edge,\edged}$ is the numerical mass flux through $\edged$ outward $D_\edge$; it must be chosen carefully to obtain the $L^2$-stability of the scheme. 
More precisely, a discrete counterpart of $\dive \bfu=0$ should be satisfied also on the dual cells. 
To define $u_{\edge,\edged}$ on internal dual edges, we distinguish two cases (see Figure \ref{fig:convection}):
\begin{list}{-}{\itemsep=0.ex \topsep=0.5ex \leftmargin=1.cm \labelwidth=0.7cm \labelsep=0.3cm \itemindent=0.cm}
\item First case -- The vector $\bfe\ei$ is normal to $\edged$, and $\edged$ is included in a primal cell $K$, 
with $\edgesi(K) = \{\edge, \edge'\}$.
Then the mass flux through $\edged=\edge | \edge'$ is given by:
\begin{equation}\label{eq:flux_eK}
|\edged|\, u_{\edge,\edged}= \frac 1 2 \ (-|\edge|\,u_{K,\edge} + |\edge'|\, u_{K,\edge'}).
\end{equation}
Note that, in this relation, all the measures of the face are the same, so this definition equivalently reads $u_{\edge,\edged}= (-u_{K,\edge}+u_{K,\edge'} )/2$.

\smallskip
\item Second case -- The vector $\bfe\ei$ is tangent to $\edged$, and $\edged$ is the union of the halves of two primal faces $\edgeperp$ and $\edgeperp'$ such that $\edge=K|L$, $\edgeperp\in \edges(K)$ and $\edgeperp' \in \edges(L)$.
The mass flux through $\edged$ is then given by:
\begin{equation}\label{eq:flux_eorth}
|\edged|\, u_{\edge,\edged} = \frac 1 2\ (|\edgeperp|\, u_{K,\edgeperp} + |\edgeperp'|\, u_{L,\edgeperp'}).
\end{equation}
\end{list}
Again, the numerical flux on a dual face is conservative:
\begin{equation} \label{conservativity}
 u_{\edge,\edged}=- u_{\edge',\edged}, \quad \mbox{for any dual face } \edged=\edge|\edge'.
\end{equation}
Moreover, if $\dive_\mesh\, \bfu = 0$, the following discrete free divergence condition holds on the dual cells:
\begin{equation} \label{cons}
\sum_{\edged\in\edgesd(D_\edge)}|\edged|\, u_{\edge,\edged}
=\frac 1 2 \sum_{\edge\in\edges(K)} |\edge|\, u_{K,\edge} +\frac 1 2 \sum_{\edge\in\edges(L)}|\edge|\, u_{L,\edge} =0.
\end{equation}
On the external dual faces associated to free degrees of freedom (which means that we are in the second of the above cases), this definition yields $u_{\edge,\edged}=0$, which is consistent with the boundary condition \eqref{bc}.

\medskip
The $i$-th component $C_\edgesi(\bfu)$ of the non linear convection operator is defined by:
\begin{equation}\label{eq:conv-opispace}
\begin{array}{l|l}
C_\edgesi(\bfu): \quad
& \quad
\Hmeshizero \longrightarrow \Hmeshizero
\\ & \displaystyle \quad 
v \longmapsto C_\edgesi(\bfu)\, v = \sum_{\edge \in \edgesinti} \frac 1{|D_\edge|} 
\sum_{\edged \in \edgesd(D_\edge)} |\edged|\, u_{\edge,\edged} v_\edged^\ast \; \characteristic_{D_\edge},
\end{array}
\end{equation}
where $v_\edged^\ast$ is chosen centred or upwind, as defined in \eqref{dual:unknown}.
The full discrete convection operator $\bfC_\edges(\bfu), \ \Hmeshzero \longrightarrow \Hmeshzero$ is defined by 
\[
\bfC_\edges (\bfu)\, \bfv = \bigl(C_{\edges^{(1)}}(\bfu)\, v_1, \ldots, C_{\edges^{(d)}}(\bfu)\, v_d \bigr).
\]
%
%--------------------------------------------------------------------------------------------------------------------------------------------------
%
\section{The steady case} \label{sec:steady}

\subsection{The scheme}

With the notations introduced in the previous sections, the MAC scheme for the discretization of the steady Navier-Stokes equations \eqref{eq:ns} on a MAC grid $(\mesh,\edges)$ reads: 
\begin{subequations}\label{eq:scheme}
\begin{align} &
\bfu \in \Hmeshzero,\ p \in L_{\mesh,0},
\\[0.5ex] &
-\Delta_\edges \bfu + {\bfC}_\edges(\bfu) \bfu + \bfnabla_\edges p = \bff,
\\[0.5ex]
& \dive_\mesh\, \bfu = 0.
\end{align} 
\end{subequations}
The discrete right-hand side of the momentum balance equation reads $\bff= \mathcal P_\edges \bar \bff$, where $\mathcal P_\edges$ is the cell mean-value operator defined by
$
\mathcal{P}_\edges \bfv =(\mathcal{P}_{\edges^{(1)}} v_1, \cdots, \mathcal{P}_{\edges^{(d)}} v_d)
\in H_{\edges^{(1)},0} \times \cdots \times H_{\edges^{(d)},0}
$
and, for $\iinud$,
\begin{equation} \label{interpedges}
\begin{array}{l|l}
\mathcal{P}_\edgesi:
&
L^1(\Omega)\longrightarrow \Hmeshizero 
\\ &
\vi\;\longmapsto \displaystyle\mathcal{P}_\edgesi \vi = \sum_{\edge \in \edgesinti} v_\edge\, \characteristic_{D_\edge}
\mbox{ with, for }\edge \in \edgesinti,\ v_\edge= \frac 1 {|D_\edge|}\int_{D_\edge} \vi(\bfx) \dx.
\end{array}
\end{equation}
Let us define the weak form $b_\edges$ of the nonlinear convection term:
\begin{multline} \label{def:weak-conv-op} \qquad
\mbox{for } (\bfu, \bfv, \bfw) \in \Hmeshzero \times \Hmeshzero \times \Hmeshzero,
\ b_\edges(\bfu, \bfv,\bfw) = \sum_{i=1}^d b_\edgesi(\bfu,\vi, \wi),
\\
\mbox{ where for } \iinud,\ b_\edgesi(\bfu,\vi, \wi) = \int_\Omega C_\edgesi(\bfu) \vi \ \wi \dx. 
\qquad \end{multline}
We can now introduce a weak formulation of the scheme, which reads:
\begin{subequations}\label{eq:weak}
\begin{align} & \nonumber
\mbox{Find } (\bfu,p) \in \Hmeshzero \times L_{\mesh,0} \mbox{ such that, for any } (\bfv,q) \in \Hmeshzero\times L_\mesh,
\\ \label{eq:mac_qdm_weak} & \hspace{15ex}
\int_\Omega \bfnabla_\edgesd \bfu : \bfnabla_\edgesd \bfv \dx + b_\edges(\bfu, \bfu, \bfv) -\int_\Omega p\, \dive_\mesh\, \bfv \dx = 
\int_\Omega \bff \cdot \bfv \dx,
\\ \label{eq:mac_mass_weak} & \hspace{15ex}
\int_\Omega \dive_\mesh\, \bfu \ q \dx =0.
\end{align} \end{subequations}
This formulation is equivalent to the strong form \eqref{eq:scheme}.

\begin{remark}[Convergence of the MAC scheme for the Stokes problem and the gradient schemes theory]
Omitting the convection terms in \eqref{eq:weak}, we obtain a weak formulation of the MAC scheme for the linear Stokes problem.
Moreover, formulating the discrete $H^1$-inner product as the integral over $\Omega$ of dot products of discrete gradients, the MAC scheme can be interpreted as a gradient scheme in the sense introduced in \cite{eym-12-sma} (see \cite{feron-eymard} and \cite{eym-15-grad} for more details on the generalization of this formulation to other schemes).
Thanks to this result, the (strong) convergence of the velocity and of its discrete gradient to the exact velocity and its gradient can be shown, and thus also the strong convergence of the pressure.
\end{remark}
%
%-----------------------------------------------------------------------------------------
%
\subsection{Stability and existence of a solution}

\begin{figure}[!t]
\begin{tikzpicture}[scale=1]
%
% \eth_1 u_1
\draw[-,fill=blue!8] (0.5,6.5)--(3.5,6.5)--(3.5,8.5)--(0.5,8.5); \path(3.,6.7) node[blue!70]{$D_\edged$}; % dualdual cell
\draw[-](0,6.5)--(4,6.5); \draw[-](0,8.5)--(4,8.5); \draw[-](0.5,6)--(0.5,9); \draw[-](3.5,6)--(3.5,9); % primal mesh
\draw[->, very thick](0.1,7.5)--(0.9,7.5); \draw[->, very thick](3.1,7.5)--(3.9,7.5);
\path(0.3,7.7) node{$u_\edge$}; \path(3.3,7.7) node{$u_{\edge'}$};
\draw[-, very thick, blue!50] (2,6.5)--(2,8.5); \path(1.85,7.5) node[blue] {$\edged$}; % dual face
\path(2,5.5) node{$(\mathcal R_\edgesd^{(1,1)} u_1)_{D_\edged} = \alpha_\edged\, u_\edge + (1-\alpha_\edged)\, u_{\edge'}$};
%
% \eth_2 u_1, internal face
\draw[-, blue!8, fill=blue!8] (1.5,1)--(3.25,1)--(3.25,3)--(1.5,3); \path(1.8,2.7) node[blue!70]{$D_\edged$}; % dualdual cell
\draw[-](0,0)--(4.5,0); \draw[-](0,2)--(4.5,2); \draw[-](0,4)--(4.5,4);
\draw[-](0.5,-0.5)--(0.5,4.5); \draw[-](2.5,-0.5)--(2.5,4.5); \draw[-](4,-0.5)--(4,4.5); % primal mesh
\draw[->, very thick](2.1,1)--(2.9,1); \draw[->, very thick](2.1,3)--(2.9,3);
\path(2.3,1.2) node{$u_\edge$}; \path(2.3,3.2) node{$u_{\edge'}$};
\draw[-, very thick, blue!50] (1.5,2)--(3.25,2); \path(2.2,2.15) node[blue] {$\edged$}; % dual face
\path(2.25,-1.2) node{$(\mathcal R_\edgesd^{(1,2)} u_1)_{D_\edged} = \alpha_\edged\,u_\edge + (1-\alpha_\edged)\, u_{\edge'}$};
%
% \eth_2 u_1, external face (up)
\draw[-, blue!8, fill=blue!8] (9.5,1)--(11.25,1)--(11.25,2)--(9.5,2); \path(9.8,1.3) node[blue!70]{$D_\edged$}; % dualdual cell
\draw[-](8,0)--(12.5,0); \draw[-, very thick](8,2)--(12.5,2);
\draw[-](8.5,-0.5)--(8.5,2); \draw[-](10.5,-0.5)--(10.5,2); \draw[-](12,-0.5)--(12,2); % primal mesh
\draw[->, very thick](10.1,1)--(10.9,1); \path(10.3,1.2) node{$u_\edge$};
\draw[-, very thick, blue!50] (9.5,2)--(11.25,2); \path(10.5,2.15) node[blue] {$\edged$}; % dual face
\path(10.25,-1.2) node{$(\mathcal R_\edgesd^{(1,2)} u_1)_{D_\edged} = \alpha_\edged\, u_\edge$};
\end{tikzpicture}
\caption{Reconstruction of the first component of the velocity, in two space dimensions.
First line: $\mathcal R_\edgesd^{(1,1)}$.
Second line: $\mathcal R_\edgesd^{(1,2)}$, inner dual face (left) and dual face lying on the boundary (right).
The real number $\alpha_\edged$ is only supposed to satisfy $\alpha_\edged\in[0,1]$.}
\label{fig:recons} \end{figure}
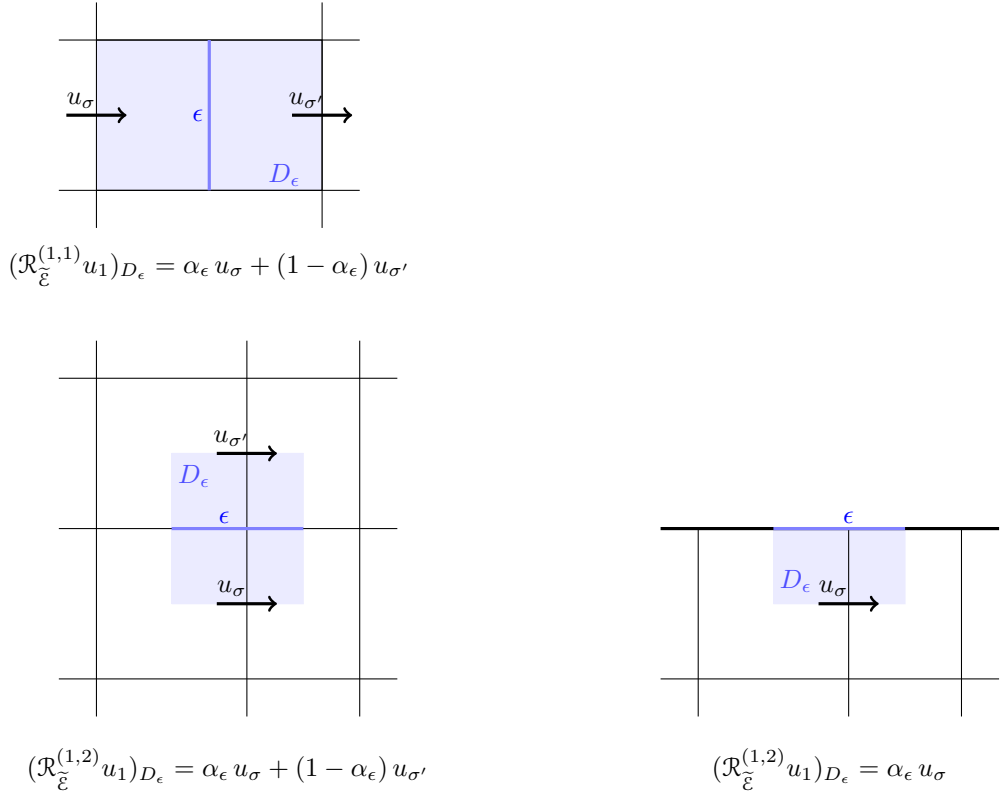

To prove the scheme stability, it is convenient to first reformulate the trilinear form associated to the velocity convection term.
To this purpose, we introduce a reconstruction of the velocity components on the partitions which where used for the definition of the discrete velocity gradient.
This leads to define $d\times d$ class of reconstruction operators, denoted by $\mathcal R_\edgesd^{(i,j)}$, with $\mathcal R_\edgesd^{(i,j)}$ acting on the $i^{th}$ component of the velocity and providing a reconstruction of this field on the partition of $\Omega$ associated to its $j^{th}$ partial derivative.

\begin{definition}[Velocity reconstructions]\label{def:vel-reco}
Let $(\mesh,\edges)$ be a given MAC mesh, and let $i,j \in \llbracket 1, d \rrbracket$.
Let $\mathcal R_\edgesd^{(i,j)}$ be a reconstruction operator defined as follows:
\[
\begin{array}{l|ccl}
\mathcal R_\edgesd^{(i,j)}: \quad & \Hmeshizero & \to & L^2(\Omega)
\\[2ex]
& v & \mapsto & \displaystyle
\mathcal R_\edgesd^{(i,j)} v =
\sum_{\edged \in \edgesdi,\ \edged \perp \bfe\ej} (\mathcal R_\edgesd^{(i,j)} v)_{D_\edged}\ \characteristic_{D_\edged},
\end{array}
\]
where $(\mathcal R_\edgesd^{(i,j)} v)_{D_\edged}$ is a convex combination of the (one of two) discrete values of the $i^{th}$ component of the velocity lying on faces of $D_\edged$ (see Figure \ref{fig:recons}).
\end{definition}

\medskip
Such a reconstruction operator satisfies the following stability result.

\begin{lemma}[Stability of the velocity reconstruction operators]\label{lem:stab-vel-reco}
Let $(\mesh,\edges)$ be a given MAC mesh, $i,j \in \llbracket 1, d\rrbracket$, and $\mathcal R_\edgesd^{(i,j)}$ be a reconstruction operator, in the sense of Definition \ref{def:vel-reco}.
Then, for $p \in [1,+\infty)$, there exists $C_{\eta_\mesh} \ge 0$, depending only on $p$ and on the parameter $\eta_\mesh$ characterizing the regularity of the mesh defined by \eqref{regmesh}, and non-decreasing with respect to $\eta_\mesh$, such that, for any $v \in \Hmeshizero$, 
\[
\| \mathcal R_\edgesd^{(i,j)} v \|_{L^p(\Omega)} \le C_{\eta_\mesh}\ \| v \|_{L^p(\Omega)}.
\]
\end{lemma}

\begin{proof}
Let $p \in [1,+\infty)$, $i,j \in \llbracket 1, d\rrbracket$ and $v \in \Hmeshizero$.
We have:
\begin{align*}
\| \mathcal R_\edgesd^{(i,j)} v \|_{L^p(\Omega)}^p
&
= \sum_{\edged \in \edgesdi,\ \edged \perp \bfe\ej} |D_\edged|\ |(\mathcal R_\edgesd^{(i,j)} v)_{D_\edged}|^p
\\ &
=\sum_{\substack{\edged \in \edgesdinti \\ \edged=\overrightarrow{\edge\edge'},\ \edged \perp \bfe\ej}}
|D_\edged|\ \bigl|\alpha_\edged v_\edge +(1-\alpha_\edged)\, v_{\edge'}\bigr|^p
+ \sum_{\substack{\edged \in \edgesdexti \\ \edged\in\edgesd(D_\edge),\ \edged \perp \bfe\ej}}
|D_\edged|\ |\alpha_\edged v_\edge|^p.
\end{align*}
Since $|\alpha_\edged|\leq 1$ and $(a+b)^p\leq 2^{p-1}(a^p+b^p)$, for $a,b \in [0,+\infty)$, we get:
\[
\| \mathcal R_\edgesd^{(i,j)} v \|_{L^p(\Omega)}^p \leq
2^{p-1} \sum_{\substack{\edged \in \edgesdinti \\ \edged=\edge\edge',\ \edged \perp \bfe\ej}}
|D_\edged|\ \bigl(|v_\edge|^p + |v_{\edge'}|^p\bigr)
+ \sum_{\substack{\edged \in \edgesdexti \\ \edged\in\edgesd(D_\edge),\ \edged \perp \bfe\ej}}
|D_\edged|\ |v_\edge|^p.
\]
Reordering the sums, we obtain that:
\[
\| \mathcal R_\edgesd^{(i,j)} v \|_{L^p(\Omega)}^p \leq
2^{p-1} \sum_{\edge\in\edgesi} |V_\edge|\ |v_\edge|^p,
\]
where the volume $V_\edge$ is the sum of the two volumes $D_\edged$ such that $\edged$ is a face of $D_\edge$.
It may now be easily checked that there exists $C_{\eta_\mesh}$ depending only on the the parameter $\eta_\mesh$ and non-decreasing with respect to this parameter such that $|V_\edge|\leq C_{\eta_\mesh} |D_\edge|$, which concludes the proof.
\end{proof}

\medskip
The discretization of the velocity convection term in the $i^{th}$ momentum balance equation may be seen as a discrete counterpart of $\dive (\vi \, \bfu)$, where $\vi$ is the convected component of the velocity field (in the scheme, $\vi=\ui$).
Multipling this expression by $\wi$ and inegrating over $\Omega$ yields a continuous counterpart $b\ei(\bfu,\vi,\wi)$ of $b_\edgesi(\bfu,\vi,\wi)$ which reads:
\[
b\ei(\bfu,\vi,\wi) =\int_\Omega \dive (\vi \, \bfu)\ \wi \dx.
\]
An integration by parts (supposing that $\wi$ vanishes on the boundary) yields:
\[
b\ei(\bfu,\vi,\wi) =- \int_\Omega \vi \, \bfu \cdot \bfnabla \wi \dx
= - \sum_{j=1}^d \int_\Omega \vi \, \uj\, \partial_j \wi \dx.
\]
The following lemma states a discrete equivalent of this relation.

\begin{lemma}[Reformulation of $b_\edges$]\label{lem:b_new_form}
Let $(\mesh,\edges)$ be a given MAC mesh, $\iinud$, and $(\bfu,\bfv,\bfw) \in \Hmeshzero \times \Hmeshzero \times \Hmeshzero$.
Let $b_\edgesi(\bfu,\vi,\wi)$ be given by \eqref{def:weak-conv-op}.
Then there exists two reconstruction operators in the sense of Definition \ref{def:vel-reco}, denoted by $(\mathcal R_{\edgesd_n}^{(i,j)})^u$ and $(\mathcal R_{\edgesd_n}^{(j,i)})^v$, such that:
\[
b_\edgesi(\bfu,\vi,\wi)=- \sum_{j=1}^d
\int_\Omega (\mathcal R_\edgesd^{(i,j)})^v \vi\ (\mathcal R_\edgesd^{(j,i)})^u \uj\ \eth_j \wi \dx.
\]
\end{lemma}

\begin{proof}
Let $(\bfu, \bfv,\bfw) \in \bfE_\edges\times\Hmeshzero^2$. 
By definition, 
\[
b_\edgesi(\bfu,\vi,\wi)
= \sum_{\edge\in\edgesi} w_\edge \sum_{\edged\in\edgesd(D_\edge)} |\edged|\ v^\ast_\edged\, u_{\edge,\edged}.
\]
Reordering the sums, we get by conservativity:
\[
b_\edgesi(\bfu,\vi,\wi)
= \sum_{\edged=\edge|\edge' \in \edgesdinti} |\edged|\ v^\ast_\edged\ u_{\edge,\edged}\ (w_\edge-w_{\edge'})
= -\sum_{\edged=\edge|\edge' \in \edgesdinti} |D_\edged|\ v^\ast_\edged\ u_{\edge,\edged}\ \dfrac{w_{\edge'}-w_\edge}{d_\edged}.
\]
The sum is over the whole set of dual faces $\edgesdi$, so over the $d$ partitions involved in the definition of the discrete gradient of $\wi$. 
In addition, without loss of generality, we may suppose that we have chosen for $\edged$ the orientation such that $\edged=\overrightarrow{\edge|\edge'}$.
Hence, we get, by definition \eqref{eq:partial-v},
\[
\dfrac{w_{\edge'}-w_\edge}{d_\edged}=(\eth_j \wi)_{D_\edged},
\]
where $j$ is the index such that $\edged$ is normal no $\bfe\ej$.
For the centered version of the convection operators, $v^\ast_\edged=(v_\edge + v_{\edge'})/2$; in the upwind case, it is equal to either $v_\edge$ or $v_{\edge'}$ (Relation \eqref{dual:unknown}).
In both cases, it is a convex combination of the two discrete values of $\vi$ lying on the faces of $D_\edged$; there exists thus an operator $\mathcal R_\edgesd^{(i,j)}$ (still with the same meaning for $j$), in the sense of Definition \ref{def:vel-reco} such that $v^\ast_\edged= (\mathcal R_\edgesd^{(i,j)} \vi)_{D_\edged}$.
Finally, from the definition of the convection operator and with the chosen orientation for $\edged$, $u_{\edge,\edged}$ is a convex combination of the two values of $\uj$ lying on the faces on $D_\edged$: either the mean value given by \eqref{eq:flux_eK}, if $j=i$, either the convex combination of \eqref{eq:flux_eorth}, if $j\neq i$.
In addition, $D_\edged$ is a volume used in the definition of the $j^{th}$ discrete partial derivative of the $i^{th}$ component, and thus also a volume used in the definition of the $i^{th}$ discrete partial derivative of the $j^{th}$ component (both partitions are the same).
So there exists one reconstruction operator $\mathcal R_\edgesd^{(j,i)}$ such that $u_{\edge,\edged}=(\mathcal R_\edgesd^{(j,i)} \uj)_{D_\edged}$, which concludes the proof.
\end{proof}

\begin{lemma}[Estimates on $b_\edges$]
Let $(\mesh,\edges)$ be a MAC grid and let $b_\edges$ be defined by \eqref{def:weak-conv-op}.
There exists $C_{\eta_\mesh} >0$, depending only and in non-decreasing way on the regularity parameter $\eta_\mesh$ of the mesh defined by \eqref{regmesh}, such that:
\begin{equation}\label{eq:contA}
\forall\ (\bfu,\bfv,\bfw) \in \Hmeshzero^3,\qquad
|b_\edges(\bfu,\bfv,\bfw)| \leq C_{\eta_\mesh}\ \|\bfu\|_{L^4(\Omega)^d}\ \|\bfv\|_{L^4(\Omega)^d}\ \|\bfw\|_{1,\edges,0}
\end{equation}
and
\begin{equation}\label{eq:contB}
\forall\ (\bfu, \bfv,\bfw) \in \Hmeshzero^3, \qquad
|b_\edges(\bfu,\bfv,\bfw)| \leq C_{\eta_\mesh}\ \|\bfu\|_{1,\edges,0}\ \|\bfv\|_{1,\edges,0}\ \|\bfw\|_{1,\edges,0}.
\end{equation}
\end{lemma}

\begin{proof}
Let $\iinud$.
Thanks to Lemma \ref{lem:b_new_form}, there exists two reconstruction operators in the sense of Definition \ref{def:vel-reco}, denoted by $(\mathcal R_\edgesd^{(i,j)})^v$ and $(\mathcal R_\edgesd^{(j,i)})^u$ such that:
\[
b_\edgesi(\bfu,\vi,\wi)=- \sum_{j=1}^d
\int_\Omega (\mathcal R_\edgesd^{(i,j)})^v \vi\ (\mathcal R_\edgesd^{(j,i)})^u \uj\ \eth_j \wi \dx.
\]
Thanks to H\"older's inequality, we get, for $j \in \llbracket 1, d \rrbracket$:
\[
\Bigl| \int_\Omega (\mathcal R_\edgesd^{(i,j)})^v \vi\ (\mathcal R_\edgesd^{(j,i)})^u \uj\ \eth_j \wi \dx \Bigr| \leq
\|(\mathcal R_\edgesd^{(i,j)})^v \vi\|_{L^4(\Omega)}
\ \|(\mathcal R_\edgesd^{(i,j)})^u \uj\|_{L^4(\Omega)}
\ \|\eth_j \wi\|_{L^2(\Omega)},
\]
which, in view of Lemma \ref{lem:stab-vel-reco} and the identity \eqref{gradient-and-innerproduct}, concludes the proof of Estimate \eqref{eq:contA}.
We then deduce \eqref{eq:contB} by the discrete Sobolev inequality \cite[Lemma 3.5]{book} which allows to control the $L^4$-norm by the discrete $H^1$-norm.
\end{proof}

\bigskip
Let us now prove that $b_\edges$ is skew-symmetrical with respect to the last two variables.
At the continuous level, this result is obtained as follows.
For $\iinud$, on one side, we have by integration by parts:
\begin{equation} \label{eq:cont-1}
\int_\Omega \wi \ \dive(\vi\,\bfu) \dx= - \int_\Omega \vi\,\bfu \cdot \bfnabla \wi \dx.
\end{equation}
On the other side, since $\dive\,\bfu=0$, $\dive(\vi\,\bfu)= \bfu \cdot \bfnabla \vi$, so:
\[
\int_\Omega \wi \ \dive(\vi\,\bfu) \dx= \int_\Omega \wi\,\bfu \cdot \bfnabla \vi \dx,
\]
which yields the conclusion.
The following lemma states a discrete analogue of this property.

\begin{lemma} [$b_\edges$ is skew-symmetrical] \label{lem:skew-sym}
Let $(\bfu, \bfv,\bfw) \in \bfE_\edges\times\Hmeshzero\times\Hmeshzero$, and let $\iinud$.
Assume the centred choice for $v_\edged$ in the expression of $b_\edgesi$; then 
\begin{equation}\label{eq:antB}
b_\edgesi(\bfu, \vi,\wi)=-b_\edgesi(\bfu, \wi,\vi),
\end{equation}
and therefore,
\begin{equation}\label{eq:buuuzero}
b_\edgesi(\bfu,\vi,\vi) = 0.
\end{equation}
Assume now the upwind choice for $v_\edged$ in the expression of $b_\edgesi$; then,
\begin{equation} \label{eq:buuupos}
b_\edgesi(\bfu,\vi,\vi) \ge 0.
\end{equation}
\end{lemma}

\begin{proof}
We mimick the computation performed in the continuous case.
At the discrete level and for the centred formulation of the convection term, we have, by a simple reordering of the sum:
\[
b_\edgesi(\bfu,\vi,\wi)
= \sum_{\edge\in\edgesinti} w_\edge \sum_{\edged=\edge|\edge'\in\edgesd(D_\edge)} |\edged|\ \frac{v_\edge+v_{\edge'}} 2 \ u_{\edge,\edged}
= \sum_{\edged=\edge|\edge' \in \edgesdinti} |\edged|\ \frac{v_\edge+v_{\edge'}} 2\ (w_\edge-w_{\edge'})\ u_{\edge,\edged}.
\]
This relation is just obtained by conservativity of the mass flux, by a process which may be seen as a discrete integration by parts, and we have seen that it may be written as a discrete analogue of \eqref{eq:cont-1} (Lemma \ref{lem:b_new_form}).
On the other hand, thanks to \eqref{cons} (\ie\ the discrete analogue of $\dive\,\bfu=0$), we have, for any face $\edge\in\edgesinti$:
\[
\sum_{\edged\in\edgesd(D_\edge)} |\edged|\ (-v_\edge)\ u_{\edge,\edged}
=-v_\edge \sum_{\edged\in\edgesd(D_\edge)} |\edged|\ u_{\edge,\edged}=0.
\]
Hence:
\begin{align*}
b_\edgesi(\bfu,\vi,\wi)
= \sum_{\edge\in\edgesinti} w_\edge \sum_{\edged=\edge|\edge'\in\edgesd(D_\edge)} |\edged|\ \frac{v_\edge+v_{\edge'}} 2 \ u_{\edge,\edged}
&
= \sum_{\edge\in\edgesinti} w_\edge \sum_{\edged=\edge|\edge'\in\edgesd(D_\edge)} |\edged|\ \frac{v_{\edge'}-v_\edge} 2 \ u_{\edge,\edged}
\\ &
= \sum_{\edged= \edge|\edge' \in \edgesdinti} |\edged|\ \frac{v_{\edge'}-v_\edge} 2\ (w_{\edge'}+w_\edge)\ u_{\edge,\edged}.
\end{align*}
This concludes the proof of \eqref{eq:antB} and \eqref{eq:buuuzero}.
In the upwind case, we have, for $\iinud$, $b_\edgesi(\bfu,\vi,\vi)=T_1+T_2$ with:
\[
T_1= \sum_{\edge\in\edgesinti} v_\edge \sum_{\edged\in\edgesd(D_\edge)} |\edged|\ v^{\mathrm c}_\edge\ u_{\edge,\edged},\qquad
T_2=\sum_{\edge\in\edgesinti} v_\edge \sum_{\edged\in\edgesd(D_\edge)} |\edged|\ (v^{\mathrm up}_\edge-v^{\mathrm c}_\edge)\ u_{\edge,\edged}.
\]
From \eqref{eq:antB}, we know that $T_1=0$.
By definition, for $\edged=\edge|\edge'$,
\[
u_{\edge,\edged}\ (v^{\mathrm up}_\edge-v^{\mathrm c}_\edge)= \frac 1 2\ u_{\edge,\edged}\ 
\begin{cases}
v_\edge-v_{\edge'} & \mbox{ if } u_{\edge,\edged}\geq 0,
\\
v_{\edge'}-v_\edge & \mbox{ if } u_{\edge,\edged}\leq 0,
\end{cases}\quad
= \frac 1 2\ |u_{\edge,\edged}|\ (v_\edge-v_{\edge'}).
\]
Thus, reordering the sums:
\[
T_2=\sum_{\edge\in\edgesinti} v_\edge \sum_{\edged=\edge|\edge'\in\edgesd(D_\edge)} |\edged|\ (v_\edge-v_{\edge'})\ |u_{\edge,\edged}|
=\sum_{\edged=\edge|\edge' \in \edgesdinti} |\edged|\ (v_\edge-v_{\edge'})^2\ |u_{\edge,\edged}|\geq 0.
\]
\end{proof}

In order to obtain an a priori estimate on the pressure, we introduce a so-called Fortin interpolation operator, \ie\ a continuous operator from $H^1_0(\Omega)^d$ to $\Hmesh$ (equipped with the discrete $H^1$-norm) which preserves the divergence. 
The following lemma is given in \cite[Theorem 1, case $q=2$]{gal-12-sta}, and we re-state it here with our notations for the sake of clarity. 
 
\begin{lemma}[Fortin interpolation operator] \label{lem:fortin}
Let $(\mesh,\edges)$ be a MAC grid of $\Omega$.
For $\bfv \in H^1_0(\Omega)^d$, we define $\widetilde{\mathcal P}_\edges \bfv$ by $\widetilde{\mathcal P}_\edges \bfv =(\widetilde{\mathcal P}_{\edges^{(1)}} v_1, \cdots, \widetilde{\mathcal P}_{\edges^{(d)}} v_d) \in \Hmesh$, where, for $\iinud$,
\begin{equation} \label{interp-moyenne}
\begin{array}{l|l}
\widetilde{\mathcal P}_\edgesi: 
&
H^1_0(\Omega) \longrightarrow \Hmeshizero
\\ & \displaystyle
\vi \longmapsto \widetilde{\mathcal P}_\edgesi \vi = \sum_{\edge \in \edgesi} v_\edge\ \characteristic_{D_\edge}
\mbox{ with } v_\edge = \frac 1 {|\edge|} \int_\edge \vi(\bfx) \dgammax,\ \edge\in\edgesi.
\end{array} \end{equation}
For $q\in L^2(\Omega)$, we define $\mathcal{P}_\mesh q \in L_\mesh$ by:
\begin{equation} \label{Pmesh}
\mathcal{P}_\mesh q (\bfx) = \frac 1 {|K|} \int_K q (\bfx) \dx.
\end{equation}
Let $\bfvarphi\in (H_0^1(\Omega))^d$.
Then:
\begin{equation} \label{conserv-div-interp}
\dive_\mesh\,(\widetilde{\mathcal P}_\edges \bfvarphi) ={\mathcal P}_\mesh (\dive \bfvarphi).
\end{equation}
In particular, if $\dive \bfvarphi=0$, then $\dive_\mesh\,(\widetilde{\mathcal P}_\edges \bfvarphi) = 0$.
In addition, there exists a real number $C_{\eta_\mesh}$, depending only on $\Omega$ and, in a non-decrasing way, on $\eta_\mesh$ defined by \eqref{regmesh}, such that:
\begin{equation} \label{norme-h1-interp}
\|\widetilde {\mathcal P}_\edges \bfvarphi\|_{1,\edges,0} \leq C_{\eta_\mesh}\ \|\bfnabla\bfvarphi\|_{L^2(\Omega)^{d \times d}}.
\end{equation}
\end{lemma}

\bigskip
\begin{theorem}[Existence and estimates] \label{prop:est}
There exists a solution to \eqref{eq:weak}, and there exists $C_{\eta_\mesh} >0$ depending only on $\Omega$ and, in a non-decreasing way, on the parameter $\eta_\mesh$ characterizing the regularity of the mesh, such that any solution of \eqref{eq:weak} satisfies the following stability estimate:
\begin{equation}
\|\bfu\|_{1,\edges,0} + \|p\|_{L^2(\Omega)} \leq C_{\eta_\mesh} \ \|\bar \bff\|_{L^2(\Omega)^d}. 
\label{estimate}
\end{equation}
\end{theorem}

\begin{proof}
Let us start by an {\it a priori} estimate on the approximate velocity. 
Assume that $(\bfu,p)\in \Hmeshzero \times L_{\mesh,0}$ satisfies \eqref{eq:scheme}; taking $\bfv=\bfu$ in \eqref{eq:mac_qdm_weak} we get that:
\[
\|\bfu\|_{1,\edges,0}^2=\int_\Omega p \ \dive_\mesh \, \bfu\dx -b_\edges(\bfu,\bfu,\bfu)
+ \int_\Omega \bff \cdot \bfu \dx.
\]
Since $\dive_\mesh\, \bfu =0$ and $b_\edges(\bfu,\bfu,\bfu)=0$, this yields that 
\begin{equation} \label{esti-vit-station}
\|\bfu\|_{1,\edges,0}\leq \diam(\Omega)\ \|\bar \bff\|_{(L^2)^d},
\end{equation}
thanks to the fact that $\| \bff\|_{(L^2(\Omega))^d} \le \|\bar \bff\|_{(L^2(\Omega))^d}$ and to the discrete Poincar\'e inequality \cite[Lemma 9.1]{book}.

\medskip 
An {\it a priori} estimate on the pressure is obtained by remarking as in \cite{shi-97-inf} that the MAC scheme is {\em inf-sup} stable, which is a consequence of the existence of a Fortin operator. 
Indeed, since $p\in L_0^2(\Omega)$, there exists $\bfvarphi\in H_0^1(\Omega)^d$ such that $\dive \bfvarphi=p$ a.e. in $\Omega$ and 
\begin{equation} \label{cNecas}
\|\bfvarphi\|_{H_0^1(\Omega)^d}\leq C_1\ \|p\|_{L^2(\Omega)},
\end{equation}
where $C_2$ depends only on $\Omega$ \cite{nec-67-met}. 
Taking $\bfv =\widetilde{\mathcal P}_\edges \bfvarphi$ (defined by \eqref{interp-moyenne}) as test function in \eqref{eq:mac_qdm_weak}, we obtain thanks to Lemma \ref{lem:fortin} that
\[
[\bfu, \bfv]_{1,\edges,0} + b_\edges(\bfu,\bfu, \bfv)-\int_\Omega p^2 \dx= \int_\Omega \bff \cdot \bfv \dx.
\]
Thanks to the estimate \eqref{eq:contB} on $b_\edges$ and the Cauchy-Schwarz inequality we get:
\[
\|p\|_{L^2(\Omega)}^2
\leq \|\bfu\|_{1,\edges,0}\|\ \bfv\|_{1,\edges,0}
+C_2\ \|\bfu\|_{1,\edges,0}^2 \|\bfv\|_{1,\edges,0} 
+\|\bar \bff\|_{L^2(\Omega)^d}\|\ \bfv\|_{L^2(\Omega)^d},
\]
where the real number $C_2$ is a non-decreasing function of $\eta_\mesh$.
This yields 
\begin{equation} \label{esti-pres-stionn:b}
\|p\|_{L^2}\leq C_{\eta_\mesh} \| \bar \bff \|_{L^2(\Omega)^d},
\end{equation}
with $C_{\eta_\mesh}$ non-decreasing with respect to $\eta_\mesh$, thanks to \eqref{norme-h1-interp}, \eqref{cNecas} and to the estimate \eqref{esti-vit-station}.

\medskip 
Let us now prove the existence of a solution to \eqref{eq:weak}.
Consider the continuous mapping:
\[
\begin{array}{l|ccl}
F: \quad & \Hmeshzero\times L_{\mesh,0}\times [0,1] & \longrightarrow & \Hmeshzero\times L_{\mesh,0},
\\[1ex] &
(\bfu,p,\zeta) & \mapsto & F(\bfu,p,\zeta) = (\hat{\bfu},\hat{p}),
\end{array}
\]
where $(\hat{\bfu},\hat{p}) \in \Hmeshzero\times L_{\mesh,0}$ is such that:
\begin{subequations}
\begin{align} \label{F1} &
\int_\Omega \hat{\bfu}\cdot\bfv \dx = [\bfu,\bfv]_{1,\edges,0}+\zeta \ b_\edges(\bfu,\bfu,\bfv)
-\int_\Omega \ p \ \dive_\mesh\, \bfv \dx - \int_\Omega \bff\cdot\bfv \dx,\quad \forall\bfv\in\Hmeshzero,
\\ \label{F2} &
\int_\Omega \hat{p} \ q \dx = \int_\Omega \dive_\mesh\, \bfu \ q \dx,\quad \forall q \in L_\mesh.
\end{align} \end{subequations}
It is easily checked that $F$ is well defined, since the values of $\hat{u}_i,\ \iinud$, and $\hat{p}$ are readily obtained by setting, for $\iinud$ and $\edge \in \edgesinti$, $\vi=\characteristic_{D_\edge}$, $\vj=0,\ j\ne i$ in \eqref{F1} and $q=\characteristic_K$ in \eqref{F2}.
We also note that the constraint $\hat{p} \in L_{\mesh,0}$ is satisfied, thanks to the boundary conditions on $\bfu$ (choose $q=1$ in \eqref{F2}). 
The mapping $F$ is continuous; moreover, if $(\bfu,p)\in \Hmeshzero\times L_{\mesh,0}$ is such that $F(\bfu,p,\zeta)=(0,0)$, then for any $ (\bfv,q) \in \Hmeshzero \times L_\mesh$,
\begin{subequations}\label{weak:ro}
\begin{align*} &
[\bfu,\bfv]_{1,\edges,0} +\zeta \ b_\edges(\bfu, \bfu, \bfv) - \int_\Omega p\, \dive_\mesh\,\bfv \dx = \int_\Omega \bff \cdot \bfv \dx,
\\ &
\int_\Omega \dive_\mesh\, \bfu\ q \dx =0.
\end{align*}
\end{subequations}
The arguments used in the above estimates on possible solutions of \eqref{eq:weak} may be used in a similar way to show that such a pair $(\bfu,p)$ is bounded independently of $\zeta$.
Since $F(\bfu,p,0)=0$ is a bijective affine function by the stability of the linear Stokes problem (see \cite{bla-05-con}), the existence of at least one solution $(\bfu,p)$ to the equation $F(\bfu,p,1)=0$, which is exactly \eqref{eq:weak}, follows by a topological degree argument (see \cite{Deimling} for the theory, \cite{eym-98-err} for the first application to a nonlinear scheme and \cite[Theorem 4.3]{eym-07-con} for an easy formulation of the result which can be used here).
\end{proof}
%
%----------------------------------------------------------------------------------------------
%
\subsection{Convergence analysis}

\begin{lemma}[Convergence of the velocity reconstructions]\label{lem:conv-vel-reco}
Let $(\mesh_n,\edges_n)_\nnn$ be a sequence of MAC meshes such that $h_{\mesh_n} \to 0$ as $\nti$~; assume that there exists $\eta >0$ such that $ \eta_{\mesh_n} \le \eta$ for any $n\in \xN$ (with $\eta_{\mesh_n} $ defined by \eqref{regmesh}).
Let $i,j \in \llbracket 1, d\rrbracket$, let $\bar v \in L^2(\Omega)$, and let $(v_n)_\nnn$ be such that $v_n \in H_{\edgesi_n,0}$ and $v_n$ converges to $\bar v$ as $\nti$ in $L^2(\Omega)$.
Let $\mathcal R_{\edgesd_n}^{(i,j)}$ be a velocity reconstruction operator, in the sense of Definition \eqref{def:vel-reco}.\\[1ex]
Then $\mathcal R_{\edgesd_n}^{(i,j)} v_n \to \bar v $ in $L^2(\Omega)$ as $\nti$.
\end{lemma}

\begin{proof}
Let $i,j \in \llbracket 1, d\rrbracket$.
Denoting $\mathcal R_{\edgesd_n}^{(i,j)}$ by $\mathcal R_n$ and $\mathcal P_{\edges\ei_n}$ (defined by \eqref{interpedges}) by $\mathcal P_n$ for short, we have, for any $\varphi \in C_c^\infty(\Omega)$:
\[
\|\mathcal R_n v_n - \bar v \|_{L^2(\Omega)} \le
\|\mathcal R_n v_n -\mathcal R_n \circ \mathcal P_n \bar v \|_{L^2(\Omega)} +
\|\mathcal R_n \circ \mathcal P_n \bar v - \mathcal R_n \circ \mathcal P_n \varphi \|_{L^2(\Omega)} +
\|\mathcal R_n \circ \mathcal P_n \varphi - \varphi\|_{L^2(\Omega)} +
\|\varphi - \bar v\|_{L^2(\Omega)}.
\]
Since $\mathcal R_n v_n = \mathcal R_n \circ \mathcal P_n v_n$, and thanks to the fact that $\| \mathcal R_n \|_{L^2(\Omega)}$ is bounded (see Lemma \ref{lem:stab-vel-reco}) and that $\mathcal P_n$ is an $L^2$-orthogonal projection, we get that there exists $C \ge 0$ such that
\[
\|\mathcal R_n v_n - \bar v \|_{L^2(\Omega)} \le
C \| v_n - \bar v \|_{L^2(\Omega)} +
C \| \bar v - \varphi \|_{L^2(\Omega)} +
\| \mathcal R_n \circ \mathcal P_n \varphi - \varphi \|_{L^2(\Omega)} +
\| \varphi - \bar v \|_{L^2(\Omega)}.
\]
Let $\varepsilon >0$.
Let us choose $\varphi \in C_c^\infty(\Omega)$ such that $\| \mathcal \varphi - \bar v \|_{L^2(\Omega)} \le \varepsilon/(C+1)$.
There exists $n_1$ such that $C \| v_n - \bar v \|_{L^2(\Omega)} \le \varepsilon, \, \forall n \ge n_1$, and there exists $n_2$ such that $\| \mathcal R_n \circ \mathcal P_n \varphi - \varphi \|_{L^2(\Omega)} \le \varepsilon, \, \forall n \ge n_2$.
Therefore $\| \mathcal R_n v_n - \bar v \|_{L^2(\Omega)} \le 3 \varepsilon$ for $n \ge \max(n_1,n_2)$, which concludes the proof.
\end{proof}

\begin{lemma}[Weak consistency of the nonlinear convection term] \label{lem:convconv}
Let $(\mesh_n,\edges_n)_\nnn$ be a sequence of meshes such that $h_{\mesh_n} \to 0$ as $n \to +\infty$~; assume that there exists $\eta >0$ such that $ \eta_{\mesh_n} \le \eta$ for any $n\in \xN$ (with $\eta_{\mesh_n} $ defined by \eqref{regmesh}).
Let $(\bfv_n)_\nnn$ and $(\bfw_n)_\nnn$ be two sequences of functions such that
\begin{list}{-}{\itemsep=0.ex \topsep=0.5ex \leftmargin=1.cm \labelwidth=0.7cm \labelsep=0.3cm \itemindent=0.cm}
\item $\bfv_n \in \Hmeshnzero$ and $\bfw_n \in \Hmeshnzero$, for $n \in \xN$,
\item the sequences $(\bfv_n)_\nnn$ and $(\bfw_n)_\nnn$ converge in $L^2(\Omega)^d$ to $\bar \bfv$ and $\bar \bfw$ respectively.
\end{list}
Let $(\Pi_{\edges_n})_\nnn$ be a family of interpolation operators satisfying \eqref{piedges} and let $\bfvarphi \in C_c^\infty(\Omega)^d$.\\[1ex]
Then
\[
b_\edges(\bfv_n, \bfw_n, \Pi_{\edges_n} \bfvarphi) \to 
b(\bar \bfv, \bar \bfw, \bfvarphi) = -\sum_{i=1}^d \int_\Omega \bar\wi\, \bar\bfv \cdot \bfnabla \varphi_i \dx
\quad \mbox{as }\nti.
\]
\end{lemma}

\begin{proof}
We have $b_{\edges_n}(\bfv_n, \bfw_n, \Pi_{\edges_n} \bfvarphi) = \sum_{i=1}^d b_\edgesi(\bfv,\wi,\Pi_\edgesi \varphi_i)$, where we have omitted the sub- and superscripts $n$ for the sake of clarity in the right-hand side of the equality, with, thanks to Lemma \ref{lem:b_new_form}:
\[
b_\edgesi(\bfv,\wi,\Pi_\edgesi \varphi_i)=- \sum_{j=1}^d
\int_\Omega (\mathcal R_\edgesd^{(i,j)})^w \wi\ (\mathcal R_\edgesd^{(j,i)})^v \vj\ \eth_j \Pi_\edgesi \varphi_i \dx,
\]
where $(\mathcal R_{\edgesd_n}^{(i,j)})^v$ and $(\mathcal R_{\edgesd_n}^{(j,i)})^w$ are two reconstruction operators, in the sense of Definition \ref{def:vel-reco}.
Thanks to the convergence properties of the reconstruction operators (Lemma \ref{lem:conv-vel-reco}) and the strong consistency of the discrete partial derivatives of the velocity (Lemma \ref{lem:grad_v_cons}), we obtain:
\[
b_\edgesi(\bfv,\wi,\Pi_\edgesi \varphi_i) \to - \sum_{j=1}^d \int_\Omega \bar \vj \ \bar \wi \ \partial_j \varphi_i \dx \quad \mbox{as} \quad \nti,
\]
which concludes the proof.
\end{proof}

\medskip
\begin{lemma}[Weak consistency of the nonlinear convection term, continued] \label{lem:convconv2}
Let $(\mesh_n,\edges_n)_\nnn$ be a sequence of meshes such that $h_{\mesh_n} \to 0$ as $n \to +\infty$~; assume that there exists $\eta >0$ such that $ \eta_{\mesh_n} \le \eta$ for any $n\in \xN$ (with $\eta_{\mesh_n} $ defined by \eqref{regmesh}).
Let $(\bfu_n)_\nnn$, $(\bfv_n)_\nnn$ and $(\bfw_n)_\nnn$ be three sequences of functions such that
\begin{list}{-}{\itemsep=0.ex \topsep=0.5ex \leftmargin=1.cm \labelwidth=0.7cm \labelsep=0.3cm \itemindent=0.cm}
\item $(\bfu_n,\bfv_n,\bfw_n) \in \Hmeshnzero^3$, for $n \in \xN$,
\item the sequences $(\bfu_n)_\nnn$ and $(\bfv_n)_\nnn$ converge in $L^p(\Omega)^d$, $1\leq p<6$, to $\bar \bfu$ and $\bar \bfv$ respectively,
\item the sequence $(\bfw_n)_\nnn$ converge in $L^p(\Omega)^d$, $1\leq p<6$, to $\bar \bfw \in H^1_0(\Omega)^d$, and $(\bfnabla_{\edgesd_n}\bfw_n)_\nnn$ converges to $\bfnabla \bar \bfw$ weakly in $L^2(\Omega)^{d \times d}$.
\end{list}
Then
\[
b_\edges(\bfu_n, \bfv_n, \bfw_n) \to b(\bar \bfv, \bar \bfw, \bar \bfw) \quad \mbox{as }\nti.
\]
\end{lemma}

\begin{proof}
Once again, we use the reformulation of the form $b_\edges$, provided by Lemma \ref{lem:b_new_form}.
Omitting sub- and superscripts $n$ for short, we have:
\begin{multline*}
b_{\edges_n}(\bfu_n, \bfv_n, \bfw_n) = \sum_{i=1}^d b_\edgesi(\bfu,\vi,\wi) \quad \mbox{with }
\\
b_\edgesi(\bfu,\vi,\wi)=- \sum_{j=1}^d
\int_\Omega (\mathcal R_\edgesd^{(i,j)})^v \vi\ (\mathcal R_\edgesd^{(j,i)})^u \uj\ \eth_j \wi \dx,
\mbox{ for }\iinud,
\end{multline*}
where $(\mathcal R_{\edgesd_n}^{(i,j)})^u$ and $(\mathcal R_{\edgesd_n}^{(j,i)})^v$ are two reconstruction operators, in the sense of Definition \ref{def:vel-reco}.
Thanks to the stability and convergence properties of the reconstruction operators (Lemma \ref{lem:stab-vel-reco} and \ref{lem:conv-vel-reco}), the sequences $((\mathcal R_\edgesd^{(j,i)})^u u_{j,n})_\nnn$ and $((\mathcal R_\edgesd^{(i,j)})^v v_{n,i})_\nnn$ are uniformly bounded in $L^p(\Omega)^d$, for $1\leq p<6$ and $i,j \in \llbracket 1,d\rrbracket$, and converge in $L^2(\Omega)^d$ to $\bar \bfu$ and $\bar \bfv$, respectively.
Hence, these sequences also converge in $L^2(\Omega)^d$, $1\leq p<6$, and the result follows thanks to the weak convergence in $L^2(\Omega)^{d\times d}$ of the partial derivatives.
\end{proof}

\medskip
\begin{lemma}[A discrete integration by parts formula] \label{lem:int_part}
Let $(\mesh,\edges)$ be a given MAC mesh, and $i,j \in \llbracket 1, d\rrbracket$.
Let $u$ and $v$ be two functions of $\Hmeshizero$.
Then there exists a reconstruction operator, in the sense of Definition \ref{def:vel-reco}, such that:
\[
\int_\Omega \eth_j u\ v \dx = -\int_\Omega \mathcal R_\edgesd^{(i,j)} u\ \eth_j v \dx.
\]
\end{lemma}
\begin{proof}
Let $i,j \in \llbracket 1, d \rrbracket$ and $(u,v)\in \Hmeshizero$.
We have, by conservativity:
\[
\int_\Omega \eth_j(u\,v) \dx
= \sum_{\substack{\edged \in \edgesdinti \\ \edged=\overrightarrow{\edge\edge'},\ \edged \perp \bfe\ej}}
|\edged|\ (u_{\edge'}v_{\edge'}-u_\edge v_\edge)
- \sum_{\substack{\edged \in \edgesdexti \\ \edged\in\edgesd(D_\edge),\ \edged \perp \bfe\ej}}
\eta_\edged\ |\edged|\ u_\edge v_\edge
=0,
\]
where $\eta_\edged=\pm 1$, depending on the relative locations of $\edge$ and $\edge$.
For any real number $\alpha_\edged \in [0,1]$, we have:
\[
u_{\edge'}v_{\edge'}-u_\edge v_\edge =
\bigl(u_{\edge'}-u_\edge\bigr) \bigl(\alpha_\edged v_{\edge'} + (1-\alpha_\edged) v_\edge \bigr)+
\bigl((1-\alpha_\edged) u_{\edge'} + \alpha_\edged u_\edge \bigr)\ \bigl(v_{\edge'}-v_\edge\bigr).
\]
We thus have $T_1+T_2$=0, with:
\begin{align*} &
T_1=
\sum_{\substack{\edged \in \edgesdinti \\ \edged=\overrightarrow{\edge\edge'},\ \edged \perp \bfe\ej}}
|\edged|\ \bigl(u_{\edge'}-u_\edge\bigr) \bigl(\alpha_\edged v_{\edge'} + (1-\alpha_\edged) v_\edge \bigr)
- \frac 1 2 \sum_{\substack{\edged \in \edgesdexti \\ \edged\in\edgesd(D_\edge),\ \edged \perp \bfe\ej}}
\eta_\edged\ |\edged|\ u_\edge v_\edge,
\\[1ex] &
T_2=
\sum_{\substack{\edged \in \edgesdinti \\ \edged=\overrightarrow{\edge\edge'},\ \edged \perp \bfe\ej}}
|\edged|\ \bigl((1-\alpha_\edged) u_{\edge'} + \alpha_\edged u_\edge \bigr)\ \bigl(v_{\edge'}-v_\edge\bigr)
- \frac 1 2 \sum_{\substack{\edged \in \edgesdexti \\ \edged\in\edgesd(D_\edge),\ \edged \perp \bfe\ej}}
\eta_\edged\ |\edged|\ u_\edge v_\edge.
\end{align*}
When $i=j$, all the dual faces are included in the domain (so the last sum vanishes).
In addition, a dual face $\edged$ is included in a cell of the primal mesh, say $K$, and $D_\edged=K$; we choose in this case $\alpha_\edged=|D_{K,\edge'}|/|K|$ and, by definition of the half-diamond cells, $1-\alpha=|D_{K,\edge}|/|K|$.
With this choice, we obtain:
\begin{equation} \label{eq:T1-T2}
T_1=\int_\Omega \eth_j u\ v \dx \quad \mbox{and}\quad T_2=\int_\Omega \mathcal R_\edgesd^{(i,j)} u\ \eth_j v \dx,
\end{equation}
with, for $\edged=\edge|\edge'$,
\[
(\mathcal R_\edgesd^{(i,i)} u)_{D_\edged}=\frac{|D_{K,\edge'}|}{|K|} u_\edge + \frac{|D_{K,\edge}|}{|K|} u_{\edge'}.
\]
Let us choose now consider the case $i\neq j$.
In this case, we choose $\alpha=|D_{\edge'}|/(2 \ |D_\edged|)$, so $1-\alpha=|D_\edge|/(2 \ |D_\edged|$, by definition of $D_\edged$, and we get \eqref{eq:T1-T2} with:
\[
\begin{array}{ll}
\mbox{for any } \edged=\edge|\edge' \in \edgesdinti,\ \edged \perp \bfe\ej,
\quad & \displaystyle
(\mathcal R_\edgesd^{(i,j)} u)_{D_\edged} = \frac 1 {|D_\edged|}\,(\frac{|D_{\edge'|}} 2 u_\edge + \frac{|D_\edge|} 2 u_{\edge'}),
\\[3ex]
\mbox{for any } \edged \in \edgesdexti \cap \edgesd(D_\edge),\ \edged \perp \bfe\ej,
\quad & \displaystyle
(\mathcal R_\edgesd^{(i,j)} u)_{D_\edged} = \frac 1 2\ u_\edge.
\end{array}
\]
\end{proof}

\bigskip
We are now in position to state and prove the convergence of the scheme.

\begin{theorem}[Convergence of the scheme, steady case]\label{theo:conv-stat}
Let $(\mesh_n,\edges_n)_\nnn$ be a sequence of meshes such that $h_{\mesh_n} \to 0$ as $n \to +\infty$~; assume that there exists $\eta >0$ such that $\eta_{\mesh_n} \le \eta$ for any $n\in \xN$ (with $\eta_{\mesh_n} $ defined by \eqref{regmesh}).
Let $(\bfu_n,p_n)$ be a solution to the MAC scheme \eqref{eq:scheme} or its weak form \eqref{eq:weak}, for $\mesh=\mesh_n$.
Then there exists $\bar \bfu \in H^1_0(\Omega)^d$ and $\bar p \in L^2(\Omega)$ such that, up to a subsequence:
\begin{list}{-}{\itemsep=0.ex \topsep=0.5ex \leftmargin=1.cm \labelwidth=0.7cm \labelsep=0.3cm \itemindent=0.cm}
\item the sequence $(\bfu_n)_\nnn$ converges to $\bar \bfu$ in $L^2(\Omega)^d$,
\item the sequence $(\bfnabla_{\edgesd_n} \bfu_n)_\nnn$ converges to $\bfnabla \bar \bfu$ in $L^2(\Omega)^{d\times d}$,
\item the sequence $(p_n)_\nnn$ converges to $\bar p$ in $L^2(\Omega)$,
\item $(\bar \bfu,\bar p)$ is a solution to the weak formulation of the steady Navier-Stokes equations \eqref{eq:cont-weak}.
\end{list}
\end{theorem}

\begin	{proof}
Thanks to the estimate \eqref{esti-vit-station} on the velocity, applying the classical estimate on the translates \cite[Theorem 14.2]{book} we obtain the existence of a subsequence of approximate solutions $(\bfu_n)_\nnn$ which converges to some $\bar \bfu \in L^2(\Omega)^d$.
From the estimates on the translates, we also get the regularity of the limit, that is $\bar \bfu \in H_0^1(\Omega)^d$.
The estimate \eqref{esti-pres-stionn:b} on the pressure then yields the weak convergence of a subsequence of $(p_n)_{n\in \xN}$ to some $\bar p$ in $L^2(\Omega)$.
Let us then pass to the limit in the scheme in order to prove its (weak) consistency.

\bigskip
\noindent {\bf Passing to the limit in the mass balance equation} --
Let $\psi\in C_c^\infty(\Omega)$.
Taking $\psi_n = \Pi_{\mesh_n} \psi$, the pointwise interpolate defined by \eqref{pimesh}, as test function in \eqref{eq:mac_mass_weak} and using \eqref{Ndiscret}, we get that:
\[
0 = \int_\Omega \dive_{\mesh_n}\, \bfu_n\ \psi_n \dx
= -\int_\Omega \bfu_n \cdot \bfnabla_{\edges_n} \psi_n \dx
=-\sum_{i=1}^d \int_\Omega u_{n,i}\ \eth_i \psi_n \dx.
\]
Therefore, thanks to Lemma \ref{lem:consgrad},
\[
0 = \lim_{n\to +\infty} -\sum_{i=1}^d \int_\Omega u_{n,i}\ \eth_i \psi_n \dx
= -\sum_{i=1}^d \int_\Omega \bar{u}_i\ \partial_i \psi \dx
= -\int_\Omega \bar{\bfu} \cdot \bfnabla \psi \dx = \int_\Omega \dive \bar{\bfu}\ \psi \dx,
\]
and therefore $\bar \bfu$ satisfies \eqref{eq:mac_mass_weak}.

\bigskip
\noindent {\bf Passing to the limit in the momentum balance equation} --
Let $\bfvarphi=(\varphi_1,\cdots,\varphi_d) \in C_c^\infty(\Omega)^d$, and take $\bfvarphi_n = \Pi_{\edges_n} \bfvarphi = (\varphi_{n,1}, \cdots, \varphi_{n,d})\in \bfH_{\edges_n,0} $ as test function in \eqref{eq:mac_qdm_weak}.
This yields:
\begin{equation}\label{passingqdm}
\int_\Omega \bfnabla_{\edgesd_n} \bfu_n : \bfnabla_{\edgesd_n} \bfvarphi_n \dx + b_\edges(\bfu_n,\bfu_n,\bfvarphi_n)
-\int_\Omega p_n \ \dive_{\mesh_n}\, \bfvarphi_n \dx =\int_\Omega \mathcal P_{\edges_n} \bar \bff\cdot\bfvarphi_n \dx.
\end{equation}
Thanks to the weak $L^2$-convergence of $p_n$ to $\overline p$ and to the uniform convergence of $\mathcal P_{\edges_n} \bar \bff$ to $\bar \bff $ and of $\dive_{\mesh_n}\,\bfvarphi_n$ to $\dive \bfvarphi$ (see Lemma \ref{lem:grad_v_cons}) as $n \to + \infty$, we have
\[
\int_\Omega \mathcal P_{\edges_n} \bar \bff\cdot\bfvarphi_n \dx \to \int_\Omega \bar \bff\cdot\bar \bfvarphi \dx \quad \mbox{and} \quad
\int_\Omega p_n \ \dive_{\mesh_n}\, \bfvarphi_n \dx \to \int_\Omega \bar{p}\ \dive\ \bar \bfvarphi \dx
\mbox{ as } n \to +\infty.
\]
From \cite[Proof of Theorem 9.1]{book}, thanks to the $L^2$-convergence of $\bfu_n$ to $\bar \bfu$, we get that, for $\iinud$,
\[
\int_\Omega \bfnabla_{\edgesd\ei_n} u_{n,i} \cdot \bfnabla_{\edgesd\ei_n} \varphi_{n,i} \dx = [u_{n,i},\varphi_{n,i}]_{1,\edges\ei_n,0}
\to -\int_\Omega \bar{u}_i\, \Delta \varphi_i \dx \text{ as } n \to + \infty.
\]
Therefore,
\[
\int_\Omega \bfnabla_{\edgesd_n} \bfu_n : \bfnabla_{\edgesd_n} \bfvarphi_n \dx
\to -\sum_{i=1}^d\int_\Omega \bar{u}_i\, \Delta \varphi_i \dx
=\int_\Omega \bfnabla \bar{\bfu} : \bfnabla \bfvarphi \dx \mbox{ as } n \to + \infty.
\]
By Lemma \ref{lem:convconv}, we have
\begin{equation} \label{qdm-b}
\lim\limits_{n \to +\infty} b_{\edges_n}(\bfu_n, \bfu_n, \bfvarphi_n) = \int_\Omega (\bar{\bfu}\cdot\bfnabla)\bar{\bfu}\cdot \bfvarphi \dx.
\end{equation}
Passing to the limit as $n \to +\infty$ in \eqref{passingqdm} thus yields that $\bar \bfu$ and $\bar p$ satisfy \eqref{eq:cont-weak}.

\bigskip
\noindent {\bf Strong convergence of $\bfnabla_{\edgesd_n} \bfu_n$ to $\bfnabla \bar \bfu$ in $L^2(\Omega)^{d \times d}$} --
The sequence $(\bfnabla_{\edgesd_n}\bfu_n)_\nnn$ is bounded in $L^2(\Omega)^{d\times d}$ and therefore, there exists $\xi\in L^2(\Omega)^{d\times d}$ and a subsequence still denoted by $(\bfnabla_{\edgesd_n}\bfu_n)_{\nnn}$ converging to $\xi$ weakly in $L^2(\Omega)^{d\times d}$.
Let $i,j \in \llbracket 1, d\rrbracket$, and let $\varphi$ be a function of $C^\infty_c(\Omega)$.
We denote by $\varphi_n$ the interpolate of $\varphi$ by the projection operator $\Pi_{\edges\ei_n}$ associated to the $i^{th}$ component of the velocity.
By Lemma \ref{lem:int_part}, we know that there exists a reconstruction operator $\mathcal R_\edgesd^{(i,j)}$, in the sense of Definition \ref{def:vel-reco}, such that:
\[
\int_\Omega \eth_j u_{n,i}\ \varphi_n \dx = -\int_\Omega \mathcal R_\edgesd^{(i,j)} u_{n,i}\ \eth_j \varphi_n \dx.
\]
By the strong convergence of $\varphi_n$ to $\varphi$, of $\mathcal R_\edgesd^{(i,j)} u_{n,i}$ to $\bar \ui$ and of $\eth_j \varphi_n$ to $\partial_j \varphi$, passing to the limit in the above relation, we get:
\[
\int_\Omega \xi_{i,j}\ \varphi \dx = -\int_\Omega \bar \ui\ \partial_j \varphi \dx.
\]
Integrating by parts in the right-hand side thanks to the regularity of $\bar \bfu$, we obtain:
\[
\int_\Omega \xi_{i,j}\ \varphi \dx = \int_\Omega \partial_j \bar \ui\ \varphi \dx.
\]
Hence, by density, $\xi= \bfnabla \bar \bfu$.
Taking $\bfvarphi_n=\bfu_n$ in \eqref{passingqdm} yields:
\[
\int \bfnabla_{\edgesd_n} \bfu_n : \bfnabla_{\edgesd_n} \bfu_n \dx \leq \int_\Omega \mathcal P_{\edges_n} \bar \bff\cdot\bfu_n \dx.
\]
Passing to the limit as $n \to +\infty$ we get that:
\[
\lim_{n \to +\infty} \| \bfnabla_{\edgesd_n}\bfu_n \|_{L^2(\Omega)^{d\times d}}^2
\leq \int_\Omega \bar \bff \cdot \bar \bfu \dx
= \| \bfnabla \bar\bfu\|_{L^2(\Omega)^{d\times d}}^2,
\]
which implies the strong convergence of the discrete gradient of the velocity.

\bigskip
\noindent {\bf Strong convergence of the pressure} --
Let $\bfvarphi_n\in H_0^1(\Omega)^d$ be such that $\dive \bfvarphi_n=p_n$ a.e. in $\Omega$ and $\|\bfvarphi_n\|_{H_0^1(\Omega)^d}\leq C\ \|p_n\|_{L^2(\Omega)},$ where $C$ depends only on $\Omega$.
Let $\bpsi_n= \widetilde{\mathcal P}_{\edges_n}\bfvarphi_n$; thanks to Lemma \ref{lem:fortin}, we have $\| \bpsi_n\|_{1,\edges_n,0} \le C\ C_{\eta_n}\, \|p_n\|_{L^2(\Omega)}$, and since $p_n \in L_{\mesh_n}$, we get that $\dive_{\mesh_n}\, \bpsi_n= p_n$.
Therefore, taking $\bpsi_n= \widetilde{\mathcal P}_{\edges_n}\bfvarphi_n$ as test function in \eqref{eq:mac_qdm_weak}, we obtain:
\begin{align*} &
\int_\Omega p_n^2\dx = \int_\Omega \bfnabla_{\edgesd_n} \bfu_n : \bfnabla_{\edgesd_n} \bpsi_n \dx + b_\edges(\bfu_n,\bfu_n,\bpsi_n)
-\int_\Omega \mathcal P_{\edges_n} \bar \bff\cdot \bpsi_n \dx,
\\[1ex] &
\| \bpsi_n \|_{1,\edges,0} \le C\ C_{\eta_n}\, \|p_n\|_{L^2(\Omega)}.
\end{align*}
From the bound on $\| \bpsi_n \|_{1,\edges,0}$, we know that $\bpsi_n$ converges to some $\bpsi \in H^1_0(\Omega)^d$ in $L^2( \Omega)^d$ and, by the same arguments as for the identification of $\xi$ with $\bfnabla \bar \bfu$, that $\bfnabla_{\edgesd_n} \bpsi_n \to \bfnabla \bpsi$ weakly in $L^2( \Omega)^{d \times d}$ as $n \to + \infty$.
In addition, we also have that $\dive \bfvarphi=p$ a.e. in $\Omega$.
By Lemma \ref{lem:convconv2}, $b_\edges(\bfu_n,\bfu_n,\bpsi_n)$ converges to $b(\bar\bfu,\bar\bfu,\bpsi)$.
Passing to the limit as $n \to +\infty$, we thus get that
\[
\lim_{n \to +\infty} \| p_n\|^2_{L^2(\Omega)}=
\int_\Omega \bfnabla \bar \bfu : \bfnabla \bpsi \dx + b(\bar\bfu,\bar \bfu, \bpsi) - \int_\Omega \bar \bff \cdot \bpsi \dx.
\]
Since $(\bar \bfu, \bar p)$ satisfies \eqref{eq:cont-weak}, this implies that $\| p_n\|_{L^2(\Omega)} \to\|\bar p\|_{L^2(\Omega)}$, which in turn yields that $p_n \to \bar p $ in $L^2(\Omega)$ as $n \to + \infty$.
\end{proof}
\begin{remark}[Uniqueness of the continuous solution and convergence of the whole sequence]\label{rem-uniq-stat}
		In the case where uniqueness of the solution is known, then a classical argument can be used to show that the whole sequence converges ; this is for instance the case for  small data, see e.g.  \cite[Theorem 1.3]{tem-77-nav}  or \cite[Theorem V.3.5]{boy-06-ele}.
\end{remark}

\section{The time-dependent case} \label{sec:unsteady}

\subsection{Time discretization}

Let us now turn to the time discretization of the problem \eqref{eq:ns:ins}.
We consider a MAC grid $(\mesh,\edges)$ of $\Omega$ in the sense of Definition \ref{def:MACgrid}, and a partition $0= t_0 < t_1 < \cdots < t_N=T$ of the time interval $(0,T)$, and, for the sake of simplicity, a constant time step $\deltat=t_{n+1}-t_n$; hence $t_n=n\, \deltat$, for $n \in \llbracket 0, N \rrbracket$.
Let $\{ u_\edge^{n+1},\ \edge\in\edges,\ n\in \llbracket 0, N-1 \rrbracket\}$ and $\{ p_K^{n+1},\ K\in\mesh,\ n \in \llbracket 0, N-1 \rrbracket \}$ be sets of discrete velocity and pressure unknowns.
For $n\in \llbracket 1, N \rrbracket$, we first define the corresponding piecewise constant space-dependent functions $\bfu=(u_1^n,\ldots,u_d^n)$ and $p^n$ by:
\[
\ui^n= \sum_{\edge \in \edgesi} u_\edge^n\, \characteristic_{D_\edge} \mbox{ for } \iinud,\quad
p^n=\sum_{K \in \mesh} p_K^n\, \characteristic_K.
\]
We enforce that $u_\edge^n=0$ for $\edge \in \edgesext$ and $n \in \llbracket 1, N \rrbracket$ (so $\ui^n \in \Hmeshizero$ and the sum in the relation above may be restricted to $\edgesinti$), and we set $\bfu^n= (u_1^n,\ldots, u_d^n) \in \Hmeshzero$.
Then, we define the discrete (time- and space-dependent) velocities and pressures functions by:
\[
\ui(\bfx,t) = \sum_{n= 0}^{N-1} \ui^{n+1}\, \characteristic_{]t_n, t_{n+1}]} \mbox{ for } \iinud,\quad
p(\bfx,t) = \sum_{n=0}^{N-1} p^{n+1}\, \characteristic_{]t_n, t_{n+1}]}.
\]
where $\characteristic_{]t_n, t_{n+1}]}$ is the characteristic function of the interval $]t_n, t_{n+1}]$. 
For $\iinud$, we denote by $X_{\edges,\deltat}\ei$ the set of such piecewise constant functions on time intervals and dual cells for the $i^{th}$ velocity component approximation, we set $\bfX_{\edges,\deltat} = \prod_{i=1}^d X_{\edges,\deltat}\ei$, and we denote by $Y_{\mesh,\deltat}$ the space of piecewise constant functions on time intervals and primal cells for the pressure approximation.
Setting 
\[
\bfu^0=\widetilde{\mathcal P}_\edges\bfu_0, \mbox{ \ie,\ for }\iinud, \quad 
\ui^0= \sum_{\edge \in \edgesinti} u_\edge^0 \,\characteristic_{D_\edge},
\mbox{ with } u_\edge^0 = \frac 1 {|\edge|} \int_\edge u_{0,i}(\bfx) \dgammax,\ \edge\in\edgesi,
\]
we define the discrete time derivative $\eth_t \bfu \in \bfX_{\edges, \deltat}$ by:
\[
\eth_t \bfu = \sum_{n=0}^{N-1} \frac{1}{\deltat} (\bfu^{n+1} - \bfu^n)\, \characteristic_{]t_n, t_{n+1}]}.
\]
Finally, we define the discrete right-hand side by:
\[
\bff \in \bfX_{\edges,\deltat},\quad
f_\edge^{n+1}=\frac 1 {\deltat\ |D_\edge|}\ \int_{t_n}^{t^{n+1}} \int_{D_\edge} \bar f_i(\bfx,t) \dx \dt,
\quad n \in \llbracket 0, N-1\rrbracket,\ \iinud,\ \edge\in\edgesinti.
\]
With these notations, the time-implicit MAC scheme for the transient Navier-Stokes reads:
\begin{subequations}\label{eq:scheme:ins}
\begin{align} \nonumber &
\mbox{{\bf Initialization} :}
\\ \label{eq:scheme-init} & \hspace{15ex}
\bfu^0= \widetilde{\mathcal{P}}_\edges \bfu_0.
\\[1ex] \nonumber & \mbox{\bf Step } n,\ n \in \llbracket 0, N-1 \rrbracket.
\mbox{ Solve for } \bfu^{n+1} \mbox{ and } p^{n+1}:
\\[0.5ex] \label{eq:scheme:ins-space} & \hspace{15ex} 
\bfu^{n+1} \in \Hmeshzero,\ p^{n+1} \in L_{\mesh,0}, 
\\[0.5ex] \label{eq:scheme:ins-vitesse} & \hspace{15ex}
\eth_t \bfu^{n+1}- \Delta_\edges \bfu^{n+1} + {\bfC}_\edges (\bfu^{n+1})\bfu^{n+1} + \bfnabla_\edges p^{n+1} = \bff^{n+1}, 
\\[0.5ex] \label{eq:scheme:ins-pression} & \hspace{15ex}
\dive_\mesh\, \bfu^{n+1} = 0.
\end{align} \end{subequations}
Step $n$, $n \in \llbracket 0, N-1 \rrbracket$, of the scheme \eqref{eq:scheme:ins} admits the following weak formulation:
\begin{multline} \label{eq:weak:ins}
\mbox{Find } \bfu^{n+1} \in \bfE_\edges \mbox{ such that, for any } \bfv\in \bfE_\edges,
\\
\int_\Omega \eth_t\bfu^{n+1} \cdot \bfv \dx + \int_\Omega \bfnabla_\edgesd \bfu^{n+1} : \bfnabla_\edgesd \bfv \dx
+ b_\edges(\bfu^{n+1}, \bfu^{n+1}, \bfv)= \int_\Omega \bff^{n+1} \cdot \bfv \dx. 
\end{multline}
The equivalence between this relation and \eqref{eq:scheme:ins-space}-\eqref{eq:scheme:ins-pression} (in the sense that \eqref{eq:weak:ins} implies the existence of a discrete pressure field such that \eqref{eq:scheme:ins-space}-\eqref{eq:scheme:ins-pression} is satisfied) is a consequence of the stability of the MAC scheme for the Stokes problem (\ie\ the fact that this scheme satisfies a discrete {\em inf-sup} condition).
%
% ----------------------------------------------------------------------------------------------
%
\subsection{Estimates on discrete solutions and existence}

Let us define the two following discrete norms for functions of space and time:
\[\begin{array}{l}
\mbox{For any } \bfv \in \bfX_{\mesh,\deltat},
\\ \displaystyle \hspace{10ex}
\|\bfv\|_{L^2(0,T;\Hmeshzero)}^2 = \sum_{n=0}^{N-1} \deltat\ \|\bfv^{n+1}\|_{1,\edges,0}^2,
\\[3ex] \displaystyle \hspace{10ex}
\|\bfv\|_{L^\infty(0,T;L^2(\Omega)^d)} = \max \Bigl\{ \|\bfv^{n+1}\|_{L^2(\Omega)^d},\ n\in \llbracket 0, N-1\rrbracket \Bigr\}.
\end{array}\]

\begin{lemma}[Existence and first estimates on the velocity]
\label{lem:est-vit-inst}
There exists at least a solution $\bfu\in \bfX_{\mesh,\deltat}$ satisfying \eqref{eq:scheme:ins}.
Furthermore, there exists $C > 0$ depending only on $\bfuini$ and $\bar \bff$ such that, for any function $\bfu\in \bfX_{\mesh,\deltat}$ satisfying \eqref{eq:scheme:ins}, the following estimates hold:
\begin{align} \label{estiLdeux} &
\|\bfu\|_{L^2(0,T;\Hmeshzero)} \leq C,
\\ \label{estiLinfiny} &
\|\bfu\|_{L^\infty(0,T;L^2(\Omega)^d)} \leq C.
\end{align}
\end{lemma}

\begin{proof}
We prove the a priori estimates \eqref{estiLdeux} and \eqref{estiLinfiny}. 
The existence of a solution then follows by a topological degree argument, as for the stationary case.

\medskip
Let $M \in \llbracket 0, N-1\rrbracket$; taking $\bfv=\bfu^{n+1}$ in \eqref{eq:weak:ins}, multiplying by $\deltat$ and summing the result over $n\in\llbracket 0, M\rrbracket$, we obtain thanks to Lemma \ref{lem:skew-sym} and to the Cauchy-Schwarz inequality:
\[
\sum_{n=0}^M \sum_{i=1}^d \sum_{\edge\in\edgesi} |D_\edge|\ u_\edge^{n+1} (u_\edge^{n+1}-u_\edge^n) 
+ \sum_{n=0}^M \deltat\ \|\bfu^{n+1}\|_{1,\edges,0}^2 
\le \sum_{n=0}^M \deltat\ \|\bff^{n+1}\|_{L^2(\Omega)^d}\ \|\bfu^{n+1}\|_{L^2(\Omega)^d}.
\]
Using the fact that for all $a, b \in \xR$, $2a(a-b) = (a-b)^2 + a^2 - b^2$ for the first term of the left-hand side and the discrete Poincar\'e and Young inequalities for the right-hand side, we get that 
\[
\|\bfu^{M+1}\|_{L^2(\Omega)^d}^2 + \sum_{n=0}^M \deltat\ \|\bfu^{n+1}\|_{1,\edges,0}^2 \le \|\bfu^0\|_{L^2(\Omega)^d}^2 + C_P^2\ \|\bff\|_{L^2(0,T;L^2(\Omega)^d)}^2,
\]
where $C_P >0$ depends only on $\Omega$.
On one hand, this inequality yields the $L^\infty$-estimate \eqref{estiLinfiny}; on the other hand, taking $M=N-1$, we get the $L^2$-estimate \eqref{estiLdeux}.
\end{proof}

\medskip
Next we turn to an estimate on the discrete time derivative. 
To this end, we introduce the following discrete dual norms on $\Hmeshzero$ and $\bfX_{\edges,\deltat}$:
\begin{equation} \label{dualnorm}
\begin{array}{l}
\displaystyle \bfv \in \Hmeshzero \mapsto \|\bfv\|_{\bfE_\edges'}=
\max \{ \Bigl| \int_\Omega \bfv \cdot\bfvarphi \dx \Bigr|~; \ \bfvarphi\in \bfE_\edges \mbox{ and } \|\bfvarphi\|_{1,\edges,0}\leq 1\},
\\[2ex] \displaystyle
\bfv \in \bfX_{\edges,\deltat} \mapsto
\|\bfv\|_{L^{4/3}(0,T;\bfE'_\edges)}=\left(\sum_{n=0}^{N-1} \deltat\ \|\bfv^{n+1}\|_{\bfE'_\edges}^{4/3} \right)^{3/4}.
\end{array} \end{equation}

\medskip
\begin{lemma}[Estimate on the dual norm of the velocity discrete time derivative] \label{lem:estderive}
Let $\bfu\in \bfX_{\edges,\deltat}$ be a solution to \eqref{eq:scheme:ins}. 
Then there exists $C > 0$ depending only on $\bfuini$, $\Omega$, $\bar \bff$ and, in a non-decreasing way, on $\eta_\mesh$, such that:
\[
\|\eth_t \bfu \|_{L^{4/3 }(0,T;\bfE'_\edges)}\leq C.
\]
\end{lemma}

\begin{proof} 
Taking $\bfv\in \bfE_\edges$ such that $\|\bfv\|_{1,\edges,0}\leq 1$ as test function in \eqref{eq:weak:ins}, we have, for $n \in \llbracket0, N-1\rrbracket$:
\[
\int_\Omega \eth_t \bfu^{n+1} \cdot \bfv \dx + \int_\Omega \bfnabla_\edgesd \bfu^{n+1} : \bfnabla_\edgesd \bfv \dx 
+ b_\edges(\bfu^{n+1},\bfu^{n+1},\bfv) = \int_\Omega \bff^{n+1} \cdot \bfv \dx.
\]
By Lemma \ref{lem:skew-sym} and thanks to the estimate \eqref{eq:contA}, we have:
\[
|b_\edges(\bfu^{n+1},\bfu^{n+1},\bfv)| \leq C_{\eta_\mesh} \|\bfu^{n+1}\|_{L^4(\Omega))^d}^2.
\]
Using the Cauchy-Schwarz inequality, we note that:
\[
\|\bfu^{n+1}\|_{L^4(\Omega)^d}^4=\int_\Omega |\bfu^{n+1}|\ |\bfu^{n+1}|^3 \dx
\le \|\bfu^{n+1}\|_{L^2(\Omega)^d} \|\bfu^{n+1}\|_{L^6(\Omega)^d}^3. 
\]
Therefore, thanks to the estimate \eqref{estiLinfiny} of Lemma \ref{lem:est-vit-inst} and to the discrete Poincar\'e inequality, there exists $\widetilde C_{\eta_\mesh}>0$ depending only on $\Omega$ and on the regularity of the mesh, such that:
\[
\int_\Omega \eth_t \bfu^{n+1} \cdot \bfv \dx \leq \widetilde C_{\eta_\mesh}\, \bigl(\|\bfu^{n+1}\|_{(L^6(\Omega))^d}^{3/2}
+\|\bfu^{n+1}\|_{1,\edges,0}+\|\bff^{n+1}\|_{(L^2(\Omega))^d}\bigr).
\]
Hence,
\begin{align*}
\|\eth_t \bfu^{n+1}\|_{\bfE'_\edges}^{4/3} 
&
\leq 9\ \widetilde C_{\eta_\mesh}^{4/3}\ \bigl( \|\bfu^{n+1}\|_{L^6(\Omega)^d}^2
+\|\bfu^{n+1}\|_{1,\edges,0}^{4/3} + \|\bff^{n+1}\|_{L^2(\Omega)^d}^{4/3}\bigr)
\\[1ex] &
\leq 9\ \widetilde C_{\eta_\mesh}^{4/3}\ \bigl( \|\bfu^{n+1}\|_{L^6(\Omega)^d}^2 
+\|\bfu^{n+1}\|_{1,\edges,0}^2 + \|\bff^{n+1}\|_{L^2(\Omega)^d}^2 + 2 \bigr).
\end{align*}
Multiplying this latter inequality by $\deltat$ and summing for $n \in \llbracket 0, N-1\rrbracket$, we get: 
\[
\|\eth_t \bfu\|_{L^{4/3}(0,T;\bfE'_\edges)}^{4/3} \leq 
9\ \widetilde C_{\eta_\mesh}^{4/3}\ \bigl( \|\bfu \|_{L^2(0,T,L^6(\Omega)^d)}^2
+\|\bfu \|_{L^2(0,T,\Hmeshzero)}^2 + \|\bar \bff\|_{L^2(0,T,L^2(\Omega)^d)}^2 + 2 T \bigr).
\]
We conclude by the discrete Sobolev inequality \cite[Lemma 3.5]{book} and thanks to the $L^2(0,T;\Hmeshzero)$-estimate \eqref{estiLdeux} of $\bfu$.
\end{proof}
%
% -----------------------------------------------------------------------------------------------------------------------
%
\subsection{Convergence analysis}

\begin{theorem}[Convergence of the scheme, time-dependent case] \label{theo:conv-instat}
Let $(\deltat_m)_{m\in \xN}$ and $(\mesh_m,\edges_m)_\mnn$ be a sequence of time steps and MAC grids (in the sense of Definition \ref{def:MACgrid}) such that $\deltat_m\rightarrow 0$ and $h_{\mesh_m} \to 0$ as $m \to +\infty$.
Assume that there exists $\eta >0$ such that $\eta_{\mesh_m} \le \eta$ for any $m\in \xN$ (with $\eta_{\mesh_m}$ defined by \eqref{regmesh}).
Let $\bfu_m$ be a solution to \eqref{eq:weak:ins} for $\deltat=\deltat_m$ and $(\mesh,\edges)=(\mesh_m,\edges_m)$.
Then there exists $\bar \bfu \in L^2(0,T; \bfE(\Omega))$ such that, up to a subsequence:
\begin{list}{-}{\itemsep=0.ex \topsep=0.5ex \leftmargin=1.cm \labelwidth=0.7cm \labelsep=0.3cm \itemindent=0.cm}
\item the sequence $(\bfu_m)_\mnn$ converges to $\bar \bfu$ in $L^{4/3}(0,T; L^2(\Omega)^d)$,
\item $\bar \bfu$ is a solution to the weak formulation \eqref{eq:cont-weak:ins}.
\item $\partial_t \bar \bfu \in L^{4/3}(0,T; E'(\Omega))$.
\end{list}
\end{theorem}

\begin{proof}
We proceed in four steps.

\medskip
\noindent {\bf First step: compactness in $L^{4/3}(0,T;L^2(\Omega)^d)$} --
The first step consists in applying the discrete Aubin-Simon theorem \ref{th-Aubin-Simon} in order to obtain the existence of a subsequence of $(\bfu_m)_{\mnn}$ which converges to $\bar \bfu$ in $L^{4/3}((0,T);L^2(\Omega)^d)$. 
In our setting, we apply Theorem \ref{th-Aubin-Simon} with $p= 4/3$; the Banach space $B$ is $L^2(\Omega)^d$, and the spaces $X_m$ and $Y_m$ consist in the space $\Hmeshmzero$ endowed with the norms defined respectively by Relations \eqref{norm} and \eqref{dualnorm}.
By \cite[Theorem 14.2]{book} and the Kolmogorov compactness theorem (see e.g. \cite[Theorem 14.1]{book}), we obtain that $(X_m,Y_m)_{m\in\xN}$ is compactly embedded in $B$ in the sense of Definition \ref{Compact-embeded}.
Let us then show that the sequence $({X_m},Y_m)_{\mnn}$ is compact-continuous in $L^2(\Omega)^d$ in the sense of Definition \ref{def:compact-continuous}.
Let $\bfv_m \in \Hmeshmzero$ such that $(\|\bfv_m\|_{1,\edges_m,0})_{\mnn}$ is bounded and $\|\bfv_m\|_{\bfE'_m} \rightarrow 0$ as $m\to +\infty$.
Assume that $\bfv_m\rightarrow \bfv$ in $(L^2(\Omega))^d$; by definition \eqref{dualnorm} of the dual norm, we have:
\[
\int_\Omega \bfv_m \cdot \bfv_m \dx \le \|\bfv_m \|_{1,\edges_m,0}\ \|\bfv_m \|_{\bfE'_m}.
\]
Passing to the limit in this inequality as $m \to \infty$, we get that $ \bfv= 0$, so that the sequence $({X_m},Y_m)_{\mnn}$ is compact-continuous in $L^2(\Omega)^d$.
We now check the three assumptions (H1), (H2) and (H3) of Theorem \ref{th-Aubin-Simon}: by Lemma \ref{lem:est-vit-inst}, the sequence $\|\bfu_m\|_{L^1(0,T;\Hmeshzero)}$ is bounded, and thanks to the discrete Poincar\'e inequality, the sequence $(\bfu_m)_\mnn$ is also bounded in $L^{4/3}(0,T;(L^2(\Omega)^d))$; 
furthermore, the sequence $\|\eth_t \bfu_m\|_{L^{4/3}(0,T;\bfE'_\edges)}$ is bounded by Lemma \ref{lem:estderive}.
Hence, Theorem \ref{th-Aubin-Simon} applies and there exists $\bar \bfu\in L^{4/3}(0,T;L^2(\Omega)^d)$ such that, up to a subsequence, \[\bfu_m\rightarrow \bar \bfu \text{ in } L^{4/3}\left(0,T;L^2(\Omega)^d\right)\mbox{ as } m \to +\infty.\]

\medskip
\noindent {\bf Step 2: Convergence in $L^2 (0,T; L^2(\Omega)^d)$} --
Thanks to Lemma \ref{lem:est-vit-inst}, the sequence $(\bfu_m)_\mnn$ is bounded in $L^\infty(0,T,L^2(\Omega)^d)$, and therefore, there exists $\hat{\bfu}\in L^\infty(0,T;L^2(\Omega)^d)$ and a subsequence $(\bfu_{\phi(m)})_{m\in\xN}$ converging to $ \hat{\bfu}$ $\star$-weakly in $L^\infty(0,T;L^2(\Omega)^d)$.
Since $\bfu_{\phi(m)}\rightarrow \bar{\bfu}$ in $L^{4/3}(0,T;L^2(\Omega)^d)$, the uniqueness of the limit in the sense of distributions implies that $\bar \bfu= \hat{\bfu}$ so that $\bar \bfu \in L^\infty(0,T;L^2(\Omega)^d)$. 
By a classical interpolation result on $L^p(0,T)$ spaces, we have:
\[
\|\bar \bfu -\bfu_m\|_{L^2(0,T; L^2(\Omega)^d)}\leq 
\|\bar \bfu -\bfu_m\|_{L^{4/3}(0,T; L^2(\Omega)^d)}^{2/3}\ \|\bar \bfu -\bfu_m\|_{L^\infty(0,T; L^2(\Omega)^d)}^{1/3},
\]
which implies that $\bfu_m$ converges towards $\bar \bfu$ in $L^2 (0,T; L^2(\Omega)^d)$ as $m$ tends to infinity.

\medskip
\noindent {\bf Step 3: Weak consistency of the scheme} --
The notion of weak consistency that we use here is the Lax-Wendroff notion: we show that if a sequence of approximate solutions of the scheme converges to some limit, then this limit is a weak solution to the original problem. 
Let us then show that $\bar \bfu$ satisfies \eqref{eq:cont-weak:ins}. 
Let $\bfvarphi \in C_c^\infty(\Omega\times[0,T))^d$, such that $\dive \bfvarphi =0$.
By Lemma \ref{lem:fortin}, we have $\dive_{\mesh_m}\, \widetilde{\mathcal P}_{\edges_m} \bfvarphi(\cdot,t_n)=0$, and so we can take $\bfvarphi_m^n = \widetilde{\mathcal P}_{\edges_m} \bfvarphi(\cdot,t_n) \in \bfE_\edges$ as test function in \eqref{eq:weak:ins}~; multiplying by $\deltat_m$ and summing for $n= \{0, \ldots, N_m-1\}$ (with $N_m \deltat_m = T$), we then get:
\begin{multline*}
\sum_{n=0}^{N_m-1} \deltat_m \Bigl(\int_\Omega \eth_t \bfu_m^{n+1} \cdot \bfvarphi_m^n \dx \dt
+ \int_\Omega \bfnabla_{\edgesd_m}\bfu_m^{n+1} : \bfnabla_{\edgesd_m}\bfvarphi_m^n \dx 
\\
+ b_{\edges_m}\bigl(\bfu_m^{n+1}, \bfu_m^{n+1},\bfvarphi_m^n \bigr)
- \int_\Omega \bff_m^{n+1} \cdot \bfvarphi_m^n \dx \Bigr) =0,
\end{multline*}
where the subscript $m$ in $\bff_m^{n+1}$ is here to recall that the discrete right-hand side is an interpolation of the continuous one, which depends on the mesh and time step.
The first term of the left-hand side reads $T_{1,m}= \sum_{i=1}^d T_{1,m,i}$ with:
\begin{align*}
T_{1,m,i}
&
= \sum_{n=0}^{N_m-1} \sum_{\edge\in\edgesi} |D_\edge| \ (u^{n+1}_{m,\edge}-u^n_{m,\edge})\ \varphi_{m,\edge}^n
\\ &
= -\sum_{n=0}^{N_m-1} \deltat \sum_{\edge\in\edgesi} |D_\edge|\ u_{m,\edge}^{n+1} \frac{\varphi_{m,\edge}^{n+1}-\varphi_{m,\edge}^n}{\deltat}
-\sum_{\edge\in\edgesi} |D_\edge|\ u_{m,\edge}^0\ \varphi_{m,\edge}^0
\\[0.5ex] &
= -\int_0^T \int_\Omega u_{m,i}(\bfx,t)\ \eth_t \varphi_{m,i}(\bfx,t) \dx\dt
-\int_\Omega \widetilde{\mathcal P}\ei_{\edges_m} \bar u_{0,i}(\bfx)\ \varphi_m^0(\bfx)\dx.
\end{align*}
We know that $u_{m,i} \to \bar \ui$ in $L^2(0,T;L^2(\Omega))$ as $m \to +\infty$.
By definition, the discrete partial derivative $\eth_t \varphi_{m,i}$ converges uniformly to $\partial_t \varphi_i$ as $m \to +\infty$.
Moreover, $\widetilde{\mathcal P}\ei_{\edges_m} \bar u_{0,i}$ converges to $\bar u_{0,i}$ in $L^q(\Omega)$ for all $q$ in $[1, 2]$, and $\varphi_{m,\edge}^0$ converges to $\varphi_i(\cdot,0)$ in $L^q(\Omega)$ for all $q$ in $[1, \infty]$.
Hence:
\begin{equation}\label{lim-ut}
T_{1,m} \to -\int_0^T \int_\Omega \bar \bfu(\bfx,t) \cdot \partial_t \bfvarphi(\bfx,t) \dx\dt 
- \int_\Omega \bar\bfu_0(\bfx) \cdot \bfvarphi(\bfx,0) \dx \text{ as } m \to +\infty.
\end{equation}
Let us then study the second term of the left-hand side. 
We have:
\[
\int_\Omega \bfnabla_{\edgesd_m} \bfu_m^{n+1} : \bfnabla_{\edgesd_m} \bfvarphi_m^n \dx
= \int_\Omega \bfnabla_{\edgesd_m} \bfu_m^{n+1} : \bfnabla_{\edgesd_m} \bfvarphi_m^{n+1} \dx
+ \int_\Omega \bfnabla_{\edgesd_m} \bfu_m^{n+1}: \bfnabla_{\edgesd_m} (\bfvarphi_m^n-\bfvarphi_m^{n+1}) \dx.
\]
By the same arguments as in the stationary case, we get that
\[
\sum_{n=0}^{N_m-1} \deltat_m \int_\Omega \bfnabla_{\edgesd_m} \bfu_m^{n+1} : \bfnabla_{\edgesd_m} \bfvarphi_m^{n+1} \dx
\to \int_0^T \int_\Omega \bfnabla \bar\bfu : \bfnabla \bfvarphi \dx \dt \mbox{ as } m \to +\infty.
\]
Moreover, thanks to the regularity of $\bfvarphi$,
\[
\int_\Omega \bfnabla_{\edgesd_m} \bfu_m^{n+1} : \bfnabla_{\edgesd_m} (\bfvarphi_m^{n+1} -\bfvarphi_m^n)\dx
\le \deltat_m\ C_\varphi\ \|\bfu_m^{n+1}\|_{1,\edges,0} 
\]
where $C_\varphi$ only depends on $\bfvarphi$. 
We thus obtain that
\[
\sum_{n=0}^{N_m-1} \deltat_m \int_\Omega \bfnabla_{\edgesd_m} \bfu_m^{n+1} : \bfnabla_{\edgesd_m} (\bfvarphi_m^{n+1} -\bfvarphi_m^n) \dx
\to 0 \mbox{ as } m \to +\infty.
\]
Similarly, we have:
\[
\sum_{n=0}^{N_m-1} \deltat_m \int_\Omega \bff_m^{n+1} \cdot (\bfvarphi_m^n - \bfvarphi_m^{n+1}) \dx 
\le \deltat_m\ C_\bfvarphi\ \| \bar\bff \|_{L^2(\Omega\times(0,T))^d} \to 0 \mbox{ as } m \to +\infty,
\]
so that
\[
\sum_{n=0}^{N_m-1} \deltat_m \int_\Omega \mathcal \bff_m^{n+1} \cdot \bfvarphi_m^n \dx
\to \int_0^T \int_\Omega \bar \bff \cdot \bfvarphi \dx\dt\mbox{ as } m \to +\infty.
\]
The convection term is dealt with by remarking that an easy adaptation of Lemma \ref{lem:convconv} to the time-dependent framework implies that 
\[
\sum_{m=0}^{N-1} \deltat_m\, b_\edges(\bfu_m^{n+1}, \bfu_m^{n+1},\bfvarphi_m^n) \to \int_0^T b(\bar \bfu, \bar \bfu, \bfvarphi) \dt \mbox{ as } \nti.
\] 
Therefore, $\bar \bfu$ is indeed a solution of \eqref{eq:cont-weak:ins}.

\medskip
\noindent {\bf Step 4: Regularity of the limit} --
Thanks to \cite[Theorems 14.1 and 14.2]{book}, the sequence of normed vector spaces $(\Hmeshmzero, \|\cdot \|_{1,\edges_m,0})_\mnn$ is $L^2(\Omega)^d$-limit-included in $H_0^1 (\Omega)^d$ in the sense of Definition \ref{def:blimit}.
We have $\bfu_m \rightarrow \bar \bfu$ in $L^2(0,T,L^2(\Omega)^d)$ as $m \to +\infty$ and $(\|\bfu_m \|_{L^2(0,T;\Hmeshmzero)})_{\mnn}$ is bounded thanks to Lemma \ref{lem:est-vit-inst}. 
Therefore Theorem \ref{reg-limit} applies, so that $\bar \bfu\in L^2(0,T;H^1_0(\Omega)^d)$; then, adapting the proof that $\dive \bar \bfu=0$ of the stationary case (see the proof of Theorem \ref{theo:conv-stat}), we get that $\bar \bfu\in L^2(0,T;\bfE(\Omega))$.

\medskip
Let us finally show that $\partial_t \bar \bfu \in L^{4/3}(0,T; E'(\Omega))$.
Let $\bfvarphi \in C_c^\infty(\Omega \times (0,T))$ such that $\dive \bfvarphi = 0$. 
Let $\bfvarphi_m \in \bfX_{\edges_m,\deltat_m}$ be defined by
\[
\bfvarphi_m^{n+1} = \frac 1 \deltat \int_{t_n}^{t_{n+1}} \widetilde {\mathcal P}_{\edges_m} \bfvarphi (\cdot, s) \,{\rm d}s
\mbox{ for } t \in [t_n,t_{n+1}[,\ n \in \llbracket 0, N-1\rrbracket.
\]
Note that, for $n \in \llbracket 0, N-1\rrbracket$, $\bfvarphi_m^{n+1}$ is discretely divergence-free, \ie\ $\bfvarphi_m^{n+1} \in \bfE_{\edges_m}$.
Thanks to Lemma \ref{lem:estderive}, there exists $C\ge 0$ depending only on $\bfu_0$, $\Omega$, $\eta$ and $\bar \bff$ such that:
\[
\int_0^T \int_\Omega \eth_t \bfu_m \cdot \bfvarphi_m \dx \dt \le C\ \|\bfvarphi_m \|_{L^4(0,T; \Hmeshzero)}.
\]
By Lemma \ref{lem:fortin}, there exists $C_2$ depending only on $\eta$ and $\Omega$, such that $ \| \bfvarphi_m \|_{L^4(0,T; \Hmeshzero)} \le C_2\| \bfvarphi \|_{L^4(0,T; \bfE(\Omega))},$ where $\bfE(\Omega)$ is endowed with the $H^1_0$ norm.
Hence, passing to the limit as $m \to +\infty$ in a similar way as for $T_{1,m}$ in Step 3, we get that
\[
\int_0^T \int_\Omega \bfu \cdot \partial_t \bfvarphi \dx \le C C_2\| \bfvarphi \|_{L^4(0,T;\bfE(\Omega))}.
\]
We then obtain that $\partial_t\bar \bfu \in L^{4/3}(0,T; \bfE'(\Omega))$ by density (see \cite[Theorem 1.6]{tem-77-nav} for the density of divergence-free regular functions in divergence-free functions of $H^1_0(\Omega)^d$).
\end{proof}

\begin{remark}[Uniqueness and convergence of the whole sequence] \label{rem-uniq-instat}
	In the case where uniqueness of the solution is known, then again the whole sequence converges ; this is for instance the case for  $d=2$, see e.g.  \cite[Theorem 3.2]{tem-77-nav},   under a small data assumption \cite[Theorem 3.7]{tem-77-nav} or under a short time assumption \cite[Theorem 3.11]{tem-77-nav}.
\end{remark}

\subsection{Case of the unsteady Stokes equations} \label{sec:stokes}

In the case of the unsteady Stokes equations, that is Problem \eqref{eq:ns:ins} where the nonlinear convection term in \eqref{qdm:ins} is omitted, stronger estimates can be obtained, which entail the weak convergence of the pressure. 
To obtain these bounds, the assumption that $\bfuini\in H^1(\Omega)^d$ and that $\dive \bfuini = 0$ plays a central role.

\medskip
Let us consider the following weak formulation of the unsteady Stokes problem:
\begin{multline} \label{stokes-weak:con}
\mbox{Find } (\bar \bfu, \bar p) \in L^2(0,T;\bfE(\Omega))\times L^2(0,T;L^2_0(\Omega))
\mbox{ such that } \forall \bfvarphi \in C_c^\infty([0,T[\times\Omega)^d,
\\
-\int_0^T \int_\Omega \bar \bfu(\bfx,t)\cdot\partial_t \bfvarphi(\bfx,t)\dx\dt - \int_\Omega \bfuini(\bfx)\cdot \bfvarphi(\bfx,0)\dx
+\int_0^T \int_\Omega \bfnabla \bar \bfu(\bfx,t): \bfnabla \bfvarphi(\bfx,t) \dx\dt
\\ 
-\int_0^T \int_\Omega \bar p\ \dive \bfvarphi\dx\dt 
=\int_0^T \int_\Omega \bar \bff(\bfx,t) \cdot \bfvarphi(\bfx,t) \dx \dt.
\end{multline}
Note that this formulation does not use divergence-free test functions as in \eqref{eq:cont-weak:ins}, so the pressure still appears.

\medskip
\noindent {\bf The scheme} --
We look for an approximation $(\bfu,p) \in \bfX_{\edges,\deltat}\times Y_{\mesh,\deltat}$ of $(\bfu,p)$ solution to the problem \eqref{stokes-weak:con};
we consider the time-implicit MAC scheme which reads: 
\begin{subequations}\label{eq:schstokes:ins}
\begin{align} \nonumber & 
\mbox{{\bf Initialization} :}
\\ \label{eq:schstokes-init} & \hspace{15ex}
\bfu^0= \widetilde{\mathcal P}_\edges \bfu_0.
\\[1ex] \nonumber &
\mbox{\bf Step } n,\ n \in \llbracket 0, N-1 \rrbracket.
\mbox{ Solve for } \bfu^{n+1} \mbox{ and } p^{n+1}:
\\[0.5ex] \label{eq:schstokes:ins-space} & \hspace{15ex}
\bfu^{n+1} \in \Hmeshzero,\ p^{n+1} \in L_{\mesh,0},
\\[0.5ex] \label{eq:schstokes:ins-vitesse} & \hspace{15ex}
\eth_t \bfu^{n+1} - \Delta_\edges \bfu^{n+1} + \bfnabla_\edges p^{n+1} = \bff^{n+1},
\\[0.5ex] \label{eq:schstokes:ins-pression} & \hspace{15ex}
\dive_\mesh\,\bfu^{n+1} = 0.
\end{align} \end{subequations}
Note that the choice of the discretization of the initial condition in \eqref{eq:schstokes-init}, together with the assumption $\dive \bfu_0=0$, implies that $\dive_\mesh\, \bfu^0=0$; this fact is important for the obtention of the estimates.
A weak formulation of \eqref{eq:schstokes:ins-space}--\eqref{eq:schstokes:ins-pression} reads:
\begin{multline} \label{stokes:weak}
\mbox{Find } (\bfu^{n+1},p^{n+1}) \in \bfE_\edges \times L_{\mesh,0} \mbox{ such that, } \forall \bfv \in \Hmeshzero,
\\
\int_\Omega \eth_t \bfu^{n+1} \ \bfv \dx + \int_\Omega \bfnabla_\edgesd \bfu^{n+1} : \bfnabla_\edgesd \bfv \dx 
- \int_\Omega p^{n+1} \dive_\mesh\, \bfv \dx
= \int_\Omega \bff^{n+1} \cdot \bfv \dx.
\end{multline}
The estimates of Lemma \ref{lem:est-vit-inst} on the approximate solutions obtained in the case of the Navier-Stokes equations are of course still valid.
However we get stronger estimates on $\eth_t \bfu$ and on $p$, as we proceed to show.

\begin{lemma}[Estimates on the velocity and its discrete time derivative] \label{lem:estideri}
Let $\bfu \in \bfX_{\edges,\deltat}$ be a solution to \eqref{eq:schstokes:ins}; then there exists $C > 0$ depending only on $\bfuini$, $\Omega$, $\bar \bff$ and, in a non-decreasing way, on $\eta_\mesh$, such that:
\begin{align} \label{estideux} &
\|\eth_t\ \bfu\|_{L^2(0,T;L^2(\Omega)^d)} \leq C,
\\[0.5ex] \label{infty} &
\|\bfu\|_{L^\infty(0,T;\Hmeshzero)}\leq C.
\end{align}
\end{lemma}

\begin{proof}
Let $\bfu^{n+1} \in \bfE_\edges$ be a solution to \eqref{stokes:weak}.
Taking $\bfv=\eth_t \bfu^{n+1}$ as test function, we get:
\begin{equation} \label{eq:dtu}
\int_\Omega (\eth_t \bfu^{n+1})^2 \dx
+ \int_\Omega \bfnabla_\edgesd \bfu^{n+1} : \bfnabla_\edgesd (\eth_t \bfu^{n+1}) \dx 
- \int_\Omega p^{n+1}\ \dive_\mesh (\eth_t \bfu^{n+1})\dx 
= \int_\Omega \bff^{n+1} \cdot \eth_t \bfu^{n+1} \dx.
\end{equation}
By linearity of the discrete time derivative and the discrete divergence operators, and thanks to \eqref{eq:schstokes:ins-pression}, we get that $ \dive_\mesh\, (\eth_t \bfu^{n+1}) = \eth_t (\dive_\mesh\, \bfu^{n+1})=0$.
Multiplying \eqref{eq:dtu} by $\deltat$ and summing the result over $n \in \llbracket 0, M \rrbracket$, for $M \in \llbracket 0, N-1 \rrbracket$, we obtain $T_1+T_2=T_3$ where
\[
T_1=\sum_{n=0}^M \deltat \int_\Omega (\eth_t \bfu^{n+1})^2\dx,\ 
T_2=\sum_{n=0}^M \deltat \int_\Omega \bfnabla_\edgesd \bfu^{n+1} : \eth_t(\bfnabla_\edgesd \bfu^{n+1}) \dx \mbox{ and }
T_3=\sum_{n=0}^M \deltat \int_\Omega \bff^{n+1} \cdot \eth_t \bfu^{n+1} \dx. 
\]
We have, by linearity of the discrete gradient operator: 
\[
T_2=\sum_{n=0}^M \bigl(\frac 1 2\ \|\bfu^{n+1}\|_{1 \edges,0}^2-\frac 1 2\ \|\bfu^n\|_{1,\edges,0}^2
+\frac 1 2\ \|\bfu^{n+1}-\bfu^n\|_{1,\edges,0}^2\bigr)
\geq \frac 1 2\ \|\bfu^{M+1}\|_{1,\edges,0}^2-\frac 1 2\ \|\bfu^0\|_{1,\edges,0}^2.
\]
By continuity of the Fortin operator, we have in addition that $|\bfu^0\|_{1,\edges,0} \leq C\ \|\bfu_0\|_{H^1(\Omega)^d}$, with $C$ depending only on $\Omega$ and (in a non-decreasing way) on $\eta_\mesh$.
Let us now turn to $T_3$.
By the Cauchy-Schwarz and the Young inequalities, we obtain:
\[
T_3 \leq \sum_{n=0}^M \deltat \bigr(\int_\Omega |\bff^{n+1}|^2 \dx\bigl)^{1/2} \bigr(\int_\Omega (\eth_t \bfu^{n+1})^2 \dx\bigl)^{1/2}
\\
\leq \frac 1 2 \sum_{n=0}^M \deltat \int_\Omega |\bff^{n+1}|^2 + \frac 1 2 \sum_{n=0}^M \deltat \int_\Omega (\eth_t \bfu^{n+1})^2 \dx,
\]
and the Cauchy-Schwarz inequality, together with the definition of $\bff$, yields for the first term at the right-hand side:
\[
\sum_{n=0}^M \deltat \int_\Omega |\bff^{n+1}|^2 \leq \|\bar \bff\|_{L^2 (0,T;L^2(\Omega)^d)}^2.
\]
Gathering the above inequalities, we get that:
\begin{equation}\label{est}
\sum_{n=0}^M \deltat \int_\Omega (\eth_t \bfu^{n+1})^2 \dx + \|\bfu^{M+1}\|_{1,\edges,0}^2
\le \|\bar \bff\|_{L^2 (0,T;L^2(\Omega)^d)}^2 + \|\bfu_0\|_{H^1(\Omega)^d}^2.
\end{equation}
This in turn yields the $L^2$-estimate \eqref{estideux} (taking $M=N-1$) on the discrete time derivative of the velocity, and the $L^\infty(H^1)$-estimate \eqref{infty} on the velocity itself.
\end{proof}

\medskip
\begin{lemma}[Estimate on the pressure] \label{lem:prs:stokes}
Let $(\bfu,p)\in \bfX_{\mesh,\deltat} \times Y_{\mesh,\deltat}$ be a solution to \eqref{eq:schstokes:ins}.
There exists $C\geq 0$ depending only on $\Omega$, $\bar \bff$ and, in a non-decreasing way, on $\eta_\mesh$, such that:
\begin{equation} \label{stokes:press}
\|p\|_{L^2(0,T;L^2(\Omega))}\leq C.
\end{equation}
\end{lemma}

\begin{proof}
We follow the same strategy as in the proof of the pressure estimate in Proposition \ref{prop:est}.
Therefore, let $\bfvarphi \in H^1_0(\Omega)^d$ be such that $\dive \bfvarphi = p^{n+1}$ and $\|\bfnabla \bfvarphi\|_{L^2(\Omega)^{d\times d}} \leq C\ \|p^{n+1}\|_{L^2(\Omega)}$, with $C$ depending only on $\Omega$.
Taking $\bfv = \widetilde{\mathcal{P}}_\edges \bfvarphi$ as test function in \eqref{stokes:weak}, we obtain, thanks to \eqref{conserv-div-interp}:
\[
\int_\Omega \eth_t \bfu^{n+1} \cdot \bfv \dx 
+ \int_\Omega \bfnabla_\edgesd \bfu^{n+1} : \bfnabla_\edgesd \bfv \dx
-\| p^{n+1} \|_{L^2(\Omega)}^2 
=\int_\Omega \bff^{n+1} \cdot \bfv \dx.
\] 
Thanks to the Cauchy-Schwarz and Poincar\'e inequalities and to the continuity of the Fortin operator $\widetilde{\mathcal{P}}_\edges$, we then get that there exists $C_{\eta_\mesh}$ depending on $\Omega$ and on the regularity of the mesh such that
\[
\|p^{n+1}\|_{L^2(\Omega)}^2 \leq C_{\eta_\mesh}
\left( \|\eth_t \bfu^{n+1}\|_{(L^2(\Omega))^d}^2 + \|\bfu^{n+1}\|_{1,\edges,0}^2 + \|\bff^{n+1}\|_{L^{2(\Omega)^d}}^2 \right).
\]
Summing this relation over $n \in \llbracket 0, N-1\rrbracket$ and multiplying by $\deltat$ yields the result thanks to \eqref{estiLdeux} and \eqref{estideux}.
\end{proof}

\medskip
\begin{theorem}[Convergence of the scheme, time-dependent Stokes problem]
Let $(\deltat)_\mnn$ and $(\mesh_m,\edges_m)_\mnn$ be a sequence of time steps and meshes such that $(\deltat)_m \to 0$ and $h_{\mesh_m} \to 0$ as $m \to +\infty$; assume that there exists $\eta >0$ such that $\eta_{\mesh_m} \le \eta$ for any $m\in \xN$ (with $\eta_{\mesh_m} $ defined by \eqref{regmesh}).
Let $(\bfu_m,p_m)$ be a solution to \eqref{eq:schstokes:ins} for $\deltat=\deltat_m$ and $\mesh=\mesh_m$.
Then there exists $(\bar \bfu, \bar p) \in L^2(0,T; \bfE(\Omega)) \times L^2(0,T;L^2(\Omega))$ such that, up to a subsequence:
\begin{list}{-}{\itemsep=0.ex \topsep=0.5ex \leftmargin=1.cm \labelwidth=0.7cm \labelsep=0.3cm \itemindent=0.cm}
\item the sequence $(\bfu_m)_\mnn$ converges to $\bar \bfu$ in $L^2(0,T; L^2(\Omega)^d)$,
\item the sequence $(p_m)_\mnn$ weakly converges to $\bar p$ in $ \in L^2(0,T;L^2(\Omega))$,
\item $(\bar \bfu,\bar p)$ is a solution to the weak formulation \eqref{stokes-weak:con}.
\end{list}
\end{theorem}

\begin{proof}
The convergence of the sequence of discrete solutions of the velocity follow from Theorem \ref{theo:conv-instat} and the weak convergence of the sequence of discrete solutions of the pressure in $L^2(0,T;L^2(\Omega))$ follow from the estimate \eqref{stokes:press}. 
Let us then show that $(\bar \bfu,\bar p)$ satisfies \eqref{stokes-weak:con}. 
Let $\bfvarphi \in C_c^\infty(\Omega\times[0,T))^d$.
Taking $\bfvarphi_m^n= \widetilde{\mathcal P}_{\edges_m} \bfvarphi(\cdot,t_n) \in \Hmeshmzero$ as test function in \eqref{stokes:weak}, multiplying by $\deltat_m$ and summing for $n \in \llbracket 0, N_m-1 \rrbracket$ (with $N_m \deltat_m = T$), we obtain:
\begin{multline*}
\sum_{n=0}^{N_m-1} \deltat_m \Bigl( \int_\Omega \eth_t \bfu_m^{n+1} \cdot \bfvarphi_m^n \dx 
+ \int_\Omega \bfnabla_{\edgesd_m} \bfu_m^{n+1} : \bfnabla_{\edgesd_m} \bfvarphi_m^n \dx
\\ 
- \int_\Omega p^{n+1}_m \dive_{\mesh_m}\, \bfvarphi_m^n \dx
- \int_\Omega \mathcal \bff_m^{n+1} \cdot \bfvarphi_m^n \dx \Bigr) =0.
\end{multline*}
Let us deal with the pressure term (all other terms of the equation can be dealt with as in the proof of Theorem \ref{theo:conv-instat}).
We have, by the divergence preservation property of the Fortin operator:
\[
\int_\Omega p^{n+1}_m \dive_{\mesh_m}\, \bfvarphi_m^n \dx
= \int_\Omega p^{n+1}_m \dive \bfvarphi(\bfx,t_n) \dx.
\]
Hence, thanks to the regularity of $\bfvarphi$ (\ie\ the fact that $|\bfvarphi(\bfx,t)- \bfvarphi(\bfx,t_n)|\leq C_\varphi\,\deltat_m$ for $\bfx \in \Omega$ and $t \in (t_n,t_{n+1})$) and the weak convergence of $p_m$ to $\bar p$,
\begin{multline*}
-\sum_{n=0}^{N_m-1} \deltat_m \int_\Omega p^{n+1}_m \dive_{\mesh_m}\, \bfvarphi_m^n \dx =
-\sum_{n=0}^{N_m-1} \deltat_m \int_\Omega p^{n+1}_m \dive \bfvarphi(\bfx,t_n) \dx
\\
\to -\int_0^T \int_\Omega \bar p\ \dive \bfvarphi(\bfx,t) \dx\dt
\mbox{ as } m \to +\infty.
\end{multline*}
\end{proof}
%
% -------------------------------------------------------------------------------------------------------------------------------------------------
%
\section{Appendix: Discrete functional analysis}\label{sec5:appendix}

\begin{definition}[Compactly embedded sequence of spaces] \label{Compact-embeded}
Let $B$ be a Banach space; a sequence $(X_m)_{\mnn}$ of Banach spaces included in $B$ is compactly embedded in $B$ if any sequence $(u_m)_{m\in \xN}$ satisfying:
\begin{itemize}
\item $u_m\in X_m$ ($\forall m\in\xN$),
\item the sequence $(\|u_m\|_{X_m})_{\mnn}$ is bounded, 
\end{itemize}
is relatively compact in $B$.
\end{definition}

\medskip
\begin{definition}[Compact-continuous sequence of spaces] \label{def:compact-continuous}
 Let B be a Banach space, and let $(X_m)_{\mnn}$ and $(Y_m)_{\mnn}$ be sequences of Banach spaces such that $X_m \subset B$ 
 for $m\in \xN$. The sequence $(X_m,Y_m)_{\mnn}$ is compact-continuous in $B$ if the
 following conditions are satified:
\begin{itemize}
\item[$\bullet$] The sequence $(X_m)_{\mnn}$ is compactly embedded in $B$ (see Definition \ref{Compact-embeded}),
\item[$\bullet$] $X_m\subset Y_m$ (for all $\mnn$),
\item[$\bullet$] if the sequence $(u_m)_{\mnn}$ is such that $u_m\in X_m$ (for all $\mnn$), $(\|u_m\|_{X_m})_{\mnn}$ is bounded and $\|u_m\|_{Y_m}\rightarrow 0$ as $m\to +\infty$, then any subsequence of $(u_m)_{\mnn}$ converging in $B$ converges to $0$ (in $B$).
\end{itemize}
\end{definition}

The following theorem is proved \cite{che-14-ext} and is a generalization of a previous work carried out in \cite{gal-12-com}.
\begin{theorem}[Aubin-Simon Theorem with a sequence of subspaces and a discrete derivative.] \label{th-Aubin-Simon}
Let $1 \leq p < \infty$, let B be a Banach space, and let $(X_m)_{\mnn}$ and $(Y_m)_{\mnn}$ be sequences of Banach spaces such that $X_m \subset B$ for $m\in \xN$. 
We assume that the sequence $(X_m,Y_m)_{\mnn}$ is compact-continuous in $B$.
Let $T>0$ and $(u^{(m)})_{\mnn}$ be a sequence of $L^{p}(0,T;B)$ satisfying the following conditions:
\begin{enumerate}
\item[$\bullet$](H1)
the sequence $(u^{(m)})_{\mnn}$ is bounded in $L^{p}(0,T;B)$.
\item[$\bullet$](H2)
the sequence $(\|u^{(m)}\|_{L^{1}(0,T;X_m)})_{\mnn}$ is bounded.
\item[$\bullet$](H3)
the sequence $(\|\eth_t u^{(m)}\|_{L^{p}(0,T;Y_m)})_{\mnn}$ is bounded.
\end{enumerate}
Then there exists $u\in L^{p}(0,T;B)$ such that, up to a subsequence, $u^{(m)}\rightarrow u$ in $L^{p}(0,T;B)$.
\end{theorem}

\medskip
\begin{definition}[$B$-limit-included] \label{def:blimit}
Let $B$ be a Banach space, $(X_m)_{m\in\xN}$ be a sequence of Banach spaces included in $B$ and $X$ be a Banach space included in $B$. 
The sequence $(X_m)_{m\in\xN}$ is $B$-limit-included in $X$ if there exists $C\in\xR$ such that if $u$ is the limit in $B$ of a subsequence of a sequence $(u_m)_{\mnn}$ verifying $u_m\in X_m$ and $\|u_m\|_{X_m}\leq 1$, then $u\in X$ and $\|u\|_{X}\leq C$.
\end{definition}

The regularity of a possible limit of approximate solutions may be proved thanks to the theorem which we recall below \cite[Theorem B1]{gal-14-ana}. 
\begin{theorem}[Regularity of the limit] \label{reg-limit}
Let $1\leq p< \infty$ and $T>0$. Let $B$ be a Banach space, $(X_m)_{m\in\xN}$ be a sequence of Banach spaces included in $B$ and $B$-limit-included in $X$ (where $X$ is a Banach space included in $B$).
Let $T>0$ and, for $m\in\xN$, Let $u_m \in L^{p}(0,T;X_m)$. 
We assume that the sequence $(\| u_m\|_{L^{p}(0,T;X_m)})_{m\in\xN}$ is bounded and that $u_m\rightarrow u$ a.e. as $m\rightarrow \infty$. 
Then $u\in L^{p}(0,T;X)$.
\end{theorem}
%
% -------------------------------------------------------------------------------------------------------------------------------------------------
%
\bibliographystyle{abbrv}
\bibliography{mac}

\begin{thebibliography}{10}

\bibitem{bla-99-err}
P.~Blanc.
\newblock Error estimate for a finite volume scheme on a {MAC} mesh for the
  {S}tokes problem.
\newblock In {\em Finite volumes for complex applications {II}}, pages
  117--124. Hermes Sci. Publ., Paris, 1999.

\bibitem{bla-05-con}
P.~Blanc.
\newblock Convergence of a finite volume scheme on a {MAC} mesh for the
  {S}tokes problem with right-hand side in {$H^{-1}$}.
\newblock In {\em Finite volumes for complex applications {IV}}, pages
  133--142. ISTE, London, 2005.

\bibitem{boy-06-ele}
F.~Boyer and P.~Fabrie.
\newblock {\em Mathematical tools for the study of the incompressible
  {N}avier-{S}tokes equations and related models}, volume 183 of {\em Applied
  Mathematical Sciences}.
\newblock Springer, New York, 2013.

\bibitem{che-14-ext}
E.~Ch{\'e}nier, R.~Eymard, T.~Gallou{\"e}t, and R.~Herbin.
\newblock An extension of the {MAC} scheme to locally refined meshes:
  convergence analysis for the full tensor time-dependent {N}avier-{S}tokes
  equations.
\newblock {\em Calcolo}, 52(1):69--107, 2015.

\bibitem{chenier}
R.~Ch{\'e}nier.
\newblock Transferts coupl\'es en convection naturelle/mixte pour des
  \'ecoulements fluides en r\'egime laminaire ou transitionnel - de la
  mod\'elisation \`a la simulation num\'erique.
\newblock H.D.R. Universit\'e Paris-Est, 12 2012.

\bibitem{cho-97-ana}
S.~H. Chou and D.~Y. Kwak.
\newblock Analysis and convergence of a {MAC}-like scheme for the generalized
  {S}tokes problem.
\newblock {\em Numer. Methods Partial Differential Equations}, 13(2):147--162,
  1997.

\bibitem{Deimling}
K.~Deimling.
\newblock {\em Nonlinear Functional Analysis}.
\newblock Springer, Berlin, 1985.

\bibitem{eym-15-grad}
R.~Eymard, P.~F\'eron, and C.~Guichard.
\newblock Gradient schemes for the incompressible steady {N}avier-{S}tokes
  problem.
\newblock In {\em 6th International conference on Approximation Methods and
  Numerical Modelling in Environment and Natural Resources}. Universit\'e de
  Pau, 6 2015.

\bibitem{eym-14-tri}
R.~Eymard, J.~Fuhrmann, and A.~Linke.
\newblock On {MAC} schemes on triangular {D}elaunay meshes, their convergence
  and application to coupled flow problems.
\newblock {\em Numer. Methods Partial Differential Equations},
  30(4):1397--1424, 2014.

\bibitem{eym-98-err}
R.~Eymard, T.~Gallou{\"e}t, M.~Ghilani, and R.~Herbin.
\newblock Error estimates for the approximate solutions of a nonlinear
  hyperbolic equation given by finite volume schemes.
\newblock {\em IMA J. Numer. Anal.}, 18(4):563--594, 1998.

\bibitem{book}
R.~Eymard, T.~Gallou{\"e}t, and R.~Herbin.
\newblock Finite volume methods.
\newblock In P.~G. Ciarlet and J.-L. Lions, editors, {\em Techniques of
  Scientific Computing, Part III}, Handbook of Numerical Analysis, VII, pages
  713--1020. North-Holland, Amsterdam, 2000.

\bibitem{eym-07-ana}
R.~Eymard, T.~Gallou{\"e}t, R.~Herbin, and J.-C. Latch{\'e}.
\newblock Analysis tools for finite volume schemes.
\newblock {\em Acta Math. Univ. Comenian. (N.S.)}, 76(1):111--136, 2007.

\bibitem{eym-12-sma}
R.~Eymard, C.~Guichard, and R.~Herbin.
\newblock Small-stencil 3d schemes for diffusive flows in porous media.
\newblock {\em M2AN Math. Model. Numer. Anal.}, 46:265--290, 2012.

\bibitem{eym-07-con}
R.~Eymard, R.~Herbin, and J.-C. Latch{\'e}.
\newblock Convergence analysis of a colocated finite volume scheme for the
  incompressible {N}avier-{S}tokes equations on general 2 or 3d meshes.
\newblock {\em SIAM J. Numer. Anal.}, 45(1):1--36, 2007.

\bibitem{feron-eymard}
P.~F\'eron and R.~Eymard.
\newblock Gradient scheme for stokes problem.
\newblock In {\em Finite volumes for complex applications {VII}}, volume~1,
  pages 265--274. Springer, London, 2014.
\newblock Finite Volumes for Complex Applications VII (FVCA VII), Berlin, June
  2014.

\bibitem{gal-12-com}
T.~. Gallou\"et and J.-C. Latch\'e.
\newblock Compactness of discrete approximate solutions to parabolic pdes -
  application to a turbulence model.
\newblock {\em Communications on Pure and Applied Analysis}, 11(6):2371 --
  2391, 2012.

\bibitem{gal-14-ana}
T.~Gallou\"et, R.~Herbin, A.~Larcher, and J.-C. Latch\'e.
\newblock Analysis of a fractional-step scheme for the p1 radiative diffusion
  model.
\newblock {\em Computational and Applied Mathematics}, pages 1--17, 2014.

\bibitem{gal-12-sta}
T.~Gallou\"et, R.~Herbin, and J.~Latch\'e.
\newblock ${W}^{1,q}$ stability of the {F}ortin operator for the {MAC} scheme.
\newblock {\em Calcolo}, 69:63--71, 2012.
\newblock see also http://hal.archives-ouvertes.fr/.

\bibitem{gir-96-fin}
V.~Girault and H.~Lopez.
\newblock Finite-element error estimates for the {MAC} scheme.
\newblock {\em IMA J. Numer. Anal.}, 16(3):247--379, 1996.

\bibitem{han-98-new}
H.~Han and X.~Wu.
\newblock A new mixed finite element formulation and the {MAC} method for the
  {S}tokes equations.
\newblock {\em SIAM J. Numer. Anal.}, 35(2):560--571 (electronic), 1998.

\bibitem{har-65-num}
F.~Harlow and J.~Welch.
\newblock Numerical calculation of time-dependent viscous incompressible flow
  of fluid with a free surface.
\newblock {\em Physics of Fluids}, 8:2182--2189, 1965.

\bibitem{kan-08-div}
G.~Kanschat.
\newblock Divergence-free discontinuous {G}alerkin schemes for the {S}tokes
  equations and the {MAC} scheme.
\newblock {\em Internat. J. Numer. Methods Fluids}, 56(7):941--950, 2008.

\bibitem{ler-34-ess}
J.~Leray and J.-L. Lions.
\newblock Essai sur les mouvements plans d’un fluide visqueux que limitent
  des parois.
\newblock {\em J. Math. Pures Appl.}, 13:331–418, 1934.

\bibitem{Lij-14-sup}
J.~Li and S.~Sun.
\newblock The superconvergence phenomenon and proof of the mac scheme for the
  stokes equations on non-uniform rectangular meshes.
\newblock {\em Journal of Scientific Computing}, pages 1--22, 2014.

\bibitem{nec-67-met}
J.~Ne{\v{c}}as.
\newblock {\em Les m\'ethodes directes en th\'eorie des \'equations
  elliptiques}.
\newblock Masson et Cie, \'Editeurs, Paris, 1967.

\bibitem{nic-92-ana}
R.~A. Nicolaides.
\newblock Analysis and convergence of the {MAC} scheme. {I}. {T}he linear
  problem.
\newblock {\em SIAM J. Numer. Anal.}, 29(6):1579--1591, 1992.

\bibitem{nic-96-ana}
R.~A. Nicolaides and X.~Wu.
\newblock Analysis and convergence of the {MAC} scheme. {II}. {N}avier-{S}tokes
  equations.
\newblock {\em Math. Comp.}, 65(213):29--44, 1996.

\bibitem{pat-80-num}
S.~V. Patankar.
\newblock {\em Numerical heat transfer and fluid flow}.
\newblock Series in computational methods in mechanics and thermal sciences.
  Taylor \& Francis, Hemisphere Publishing Corporation, Washington, 1980.

\bibitem{por-78-err}
T.~A. Porsching.
\newblock Error estimates for {MAC}-like approximations to the linear
  {N}avier-{S}tokes equations.
\newblock {\em Numer. Math.}, 29(3):291--306, 1977/78.

\bibitem{shi-97-inf}
D.~Shin and J.~C. Strikwerda.
\newblock Inf-sup conditions for finite-difference approximations of the
  {S}tokes equations.
\newblock {\em J. Austral. Math. Soc. Ser. B}, 39(1):121--134, 1997.

\bibitem{tem-77-nav}
R.~Temam.
\newblock {\em {N}avier-{S}tokes equations}.
\newblock Studies in mathematics and its applications. North-Holland, 1977.

\bibitem{tem-84-nav}
R.~Temam.
\newblock {\em Navier-{S}tokes equations}, volume~2 of {\em Studies in
  Mathematics and its Applications}.
\newblock North-Holland Publishing Co., Amsterdam, third edition, 1984.
\newblock Theory and numerical analysis, With an appendix by F. Thomasset.

\bibitem{wes-01-pri}
P.~Wesseling.
\newblock {\em Principles of Computational Fluid Dynamics}.
\newblock Springer, 2001.

\end{thebibliography}
\end{document}